\documentclass[11pt]{amsart}
\usepackage[T1]{fontenc}
\usepackage{txfonts}
\usepackage{float}
\usepackage{amssymb,verbatim}
\usepackage[all]{xy}
\usepackage{subfig}
\usepackage{tikz}
\usepackage{color}
\usepackage{a4wide}
\usepackage{mathrsfs}
\usepackage{longtable}
\usepackage{hyperref}
\usepackage{wrapfig}
\usetikzlibrary{arrows}
\DeclareMathAlphabet{\mathpzc}{OT1}{pzc}{m}{it}

\newcommand{\R}{\mathbb{R}}
\newcommand{\C}{\mathbb{C}}
\newcommand\Z{\mathbb{Z}}
\newcommand{\N}{\mathbb{N}}
\newcommand{\Q}{\mathbb{Q}}

\newcommand{\Lb}{\mathbf{L}}

\newcommand{\Pb}{\mathbb{P}}

\newcommand{\T}{\mathrm{T}}

\newcommand{\Ub}{\mathbf{U}}

\newcommand{\Wb}{\mathbf{W}}

\newcommand{\pp}{\mathbf{p}}

\newcommand{\xx}{\mathbf{x}}
\newcommand{\yy}{\mathbf{y}}

\newcommand{\Dcal}{\mathcal{D}}
\newcommand{\Ecal}{\mathcal{E}}

\newcommand{\Hcal}{\mathcal{H}}
\newcommand{\Ical}{\mathcal{I}}

\newcommand{\Lcal}{\mathcal{L}}
\newcommand{\Mcal}{\mathcal{M}}

\newcommand{\Ocal}{\mathcal{O}}
\newcommand{\Pcal}{\mathcal{P}}

\newcommand{\Scal}{\mathcal{S}}

\newcommand{\Ucal}{\mathcal{U}}

\newcommand{\Cc}{\mathcal{C}}
\newcommand{\Dc}{\mathcal{D}}
\newcommand{\Ec}{\mathcal{E}}
\newcommand{\Hc}{\mathcal{H}}
\newcommand{\Ic}{\mathcal{I}}
\newcommand{\Jc}{\mathcal{J}}

\newcommand{\Lc}{\mathcal{L}}

\newcommand{\Pc}{\mathcal{P}}

\newcommand{\Sc}{\mathcal{S}}
\newcommand{\Tc}{\mathcal{T}}
\newcommand{\Uc}{\mathcal{U}}

\renewcommand{\AA}{\mathscr{A}}
\newcommand{\CC}{\mathscr{C}}

\newcommand{\OO}{\mathscr{O}}

\newcommand{\card}{{\rm card}}

\newcommand{\id}{\mathrm{id}}
\newcommand{\Id}{\mathrm{Id}}

\newcommand{\CP}{\mathbb{P}}
\newcommand{\SL}{{\rm SL}}

\newcommand{\Mod}{\mathcal{M}}

\newcommand{\Aa}{\textrm{Area}}

\newcommand{\vol}{{\rm vol}}
\newcommand{\Vol}{{\rm Vol}}

\newcommand{\PSL}{\mathrm{PSL}}

\newcommand{\imag}{\mathrm{Im}}

\newcommand{\hx}{\hat{x}}
\newcommand{\hX}{\hat{X}}
\newcommand{\hZ}{\hat{Z}}

\newcommand{\homg}{\hat{\omega}}
\newcommand{\vide}{\varnothing}
\newcommand{\Sig}{\Sigma}

\newcommand{\eps}{\epsilon}

\newcommand{\ol}{\overline}
\newcommand{\ul}{\underline}

\newcommand{\ra}{\rightarrow}

\newcommand\cal{\mathcal}
\newcommand{\strate}{\Omega^d\Mcal_{g,n}(\kappa)}
\newcommand{\pstrate}{\mathbb{P}\Omega^d\Mcal_{g,n}(\kappa)}
\newcommand{\stratei}{\Omega^d_1\mathcal{M}_{g,n}(\kappa)}

\newcommand{\istrate}{\Omega^d_1\mathcal{M}_{g,n}(\kappa)}

\newcommand{\stratesp}{\Omega^d\Mcal_{0,n}(\kappa)}
\newcommand{\pstratesp}{\mathbb{P}\Omega^d\Mcal_{0,n}(\kappa)}

\newcommand{\clpstratesp}{\Pb\Omega^d\ol{\Mod}_{0,n}(\kappa)}
\newcommand{\blowupsp}{\widehat{\Mod}_{0,n}(\kappa)}

\newcommand{\hodge}{\mathcal{H}_{g,n}}
\newcommand{\khodge}{\mathcal{H}^{(d)}_{g,n}}

\newcommand{\khodgespext}{\overline{\mathcal{H}}^{(d)}_{0,n}}

\newcommand{\pkhodgespext}{\mathbb{P}\overline{\mathcal{H}}^{(d)}_{0,n}}

\newtheorem{Theorem}{Theorem}[section]
\newtheorem{Corollary}[Theorem]{Corollary}
\newtheorem{Lemma}[Theorem]{Lemma}
\newtheorem{Proposition}[Theorem]{Proposition}

\newtheorem{Definition}[Theorem]{Definition}

\newtheorem{Claim}[Theorem]{Claim}

\theoremstyle{remark}
\theoremstyle{example}
\newtheorem{Remark}[Theorem]{Remark}
\begin{document}
\title[Volumes of moduli spaces of flat surfaces in genus 0]{Intersection theory and volumes of moduli spaces of flat metrics on the sphere}

\author[D.-M. Nguyen]{Duc-Manh Nguyen \\  {(\tiny with an appendix by}  Vincent Koziarz {\tiny and } Duc-Manh Nguyen)}

\address{Univ. Bordeaux, CNRS, Bordeaux INP, IMB, UMR 5251, F-33405 Talence, France}
\email[D.-M.~Nguyen]{duc-manh.nguyen@math.u-bordeaux.fr}

\email[V.~Koziarz]{vincent.koziarz@math.u-bordeaux.fr}

\date{\today}
\begin{abstract}
Let $\mathbb{P}\Omega^d\mathcal{M}_{0,n}(\kappa)$, where $\kappa=(k_1,\dots,k_n)$, be a stratum of (projectivized) $d$-differentials in genus $0$. We prove a recursive formula which relates the volume of $\mathbb{P}\Omega^d\mathcal{M}_{0,n}(\kappa)$ to the volumes of other strata of lower dimensions in the case where none of the $k_i$ is divisible by $d$. As an application, we give a new proof of the Kontsevich's formula for the volumes of strata of quadratic differentials with simple poles and zeros of odd order, which was originally proved by Athreya-Eskin-Zorich. In another application, we show that up to some power of $\pi$, the volume of the moduli spaces of flat metrics on the sphere with prescribed cone angles is a  continuous piecewise polynomial with rational coefficients function of the angles,  provided none of the angles is an integral multiple of $2\pi$. This generalizes the results of \cite{McM_conegb} and \cite{KN18}.
\end{abstract}

\maketitle


%

\section{Introduction}\label{sec:intro}
\subsection{Volumes on moduli spaces of flat surfaces}\label{subsec:intro:flat:surf}
In this paper, by a {\em flat surface} we will mean a compact, closed, oriented  surface endowed with an Euclidean metric structure outside a finite subset called singularities such that each singularity has a neighborhood isometric to an Euclidean cone. The conformal structure induced by the flat metric turns the underlying surface into a Riemann surface with marked points at the singularities.
By a result of  Troyanov~\cite{Troy91}, there is a bijection between the set of flat metrics on a fixed compact closed surface $S$ having singularities with prescribed cone angles in a finite subset $\Sigma \subset S$  up to rescaling, and the moduli space of Riemann surfaces of genus $g$ with $|\Sigma|$ marked points. This bijection makes it relevant to use  flat  surfaces as a tool to investigate the moduli space of Riemann surfaces with marked points.

\medskip

For any $d \in \N, \; d\geq 1$, a $d$-differential on a Riemann surface $X$ is a meromorphic section of the line bundle $K_X^{\otimes d}$, where $K_X$ is the canonical line bundle of $X$.
It is a well known fact that every $d$-differential with poles of order at most $d-1$ defines a flat surface  with conical singularities at its zeros and poles. A zero of order $k \geq 1-d$ of the $d$-differential corresponds to a singular point with cone angle $2\pi(1+\frac{k}{d})$.
Let $\khodge$ denote the set of all tuples $(X,x_1,\dots,x_n,q)$, where $X$ is a Riemann surface of genus $g$, $\{x_1,\dots,x_n\}$ is a set of $n$ marked points on $X$, and $q$ is a $d$-differential on $X$ which has  poles of order at most $d-1$ and is holomorphic out side the set $\{x_1,\dots,x_n\}$. The spaces  $\khodge$'s are actually  vector bundles over $\Mod_{g,n}$.  When $d=1$, $\hodge=\mathcal{H}^{(1)}_{g,n}$ is the classical Hodge bundle, and when $d=2$, $\mathcal{H}^{(2)}_{g,n}$ is the holomorphic cotangent bundle of $\Mod_{g,n}$.
Flat surfaces defined by elements of $\hodge$ are called {\em translation surfaces}, and flat surfaces defined by elements of $\mathcal{H}^{(2)}_{g,n}$ are called {\em half-translation surfaces}.

A stratum of  $\khodge$ is a subset $\strate$ which consists of $d$-differentials whose orders at the poles and zeros are prescribed by a vector $\kappa\in (\Z_{\geq 1-d})^m$. Note that we always have $m\geq n$ since all the marked points  are considered as zeros (or poles) of the differentials.
In the case $d\in \{1,2\}$, using flat metric structures, one can see that there is a natural action of $\SL(2,\R)$ (resp. of $\PSL(2,\R)$) on  $\mathcal{H}_{g,n}$ (resp. of $\mathcal{H}^{(2)}_{g,n}$) that preserves the strata.
More interestingly, these actions are intimately related to other dynamical systems such as: billiards in rational polygons, interval exchange transformations, Teichmuller geodesic flow in moduli spaces.... We refer to  the excellent surveys \cite{MT02, Z:survey, W:survey} for more details on the connection between those dynamical systems and related topics.

Each stratum of $\hodge$ and of $\mathcal{H}^{(2)}_{g,n}$ carries a natural volume form called {\em Masur-Veech volume}. It is a classical result due to Masur and Veech that the volume of the set of surfaces with area at most one with respect to this volume form is finite. Those volumes are a crucial piece in the computation of dynamical invariants such that the Lyapunov exponents of the Teichmuller geodesic flows, or the Siegel-Veech constants (see~\cite{EMZ03, MZ08, EKZ14}). In \cite{Ng19} the result of Masur and Veech has been generalized to all $d\geq 1$  (see~\textsection~\ref{sec:loc:chart:vol:form} for more details).

\medskip

Flat surfaces also  naturally arise in a different context as follows. Let $S$ be a compact closed oriented surface. Consider a tiling $\Gamma$ of $S$ either by topological triangles or quadrangles. If one endows each tile of $\Gamma$ with the metric of an equilateral triangles of size $1$ (in the case of triangulation), or with the metric of a unit square (in the case of quadrangulation), the surface $S$ then becomes a flat surface with conical singularities. It is not difficult to see that the flat metric obtained from this construction is induced by some $d$-differential with $d\in \{1,2,3,6\}$ in the case of triangulation, and $d\in \{1,2,4\}$ in the case of quadrangulation.
Moreover, the $d$-differentials arising from tilings of $S$ by  triangles  or squares belong to some lattices in the corresponding moduli space. Thus counting the number of triangulations (resp. of quadrangulations) with at most $k$ tiles  is the same as the counting the number of $d$-differentials in those lattices with area at most $\frac{\sqrt{3}}{4}\cdot k$ (resp. at most $k$). Based on this observation, by showing that the volume of the  corresponding moduli space is finite,   Thurston~\cite{Th98} proved that the growth rate of the number of triangulations of the sphere by at most $k$ triangles such that every vertex is contained in at most $6$ triangles is  $O(k^{10})$ (see~\cite{ES18, Engel-I} for related results).

\medskip

One can adapt the definition of the Masur-Veech volume on strata of translation surfaces to obtained a volume form on every stratum of $\khodge$ for $d\in \{3,4,6\}$. Let $\stratei$ denote the set of $d$-differentials in $\strate$ with area at most $1$, and $c(\kappa)$ the Masur-Veech volume of $\stratei$. Then as $k\to +\infty$, the number of triangulations (resp. of quadrangulations) with at most $k$ tiles that give rise to $d$-differentials in $\strate$ is equal to $\lambda\cdot c(\kappa)\cdot k^{s}+o(k^s)$, where $\lambda$ is a universal constant and $s=\dim_\C\strate$. Therefore, the asymptotics of the number of triangulations/quadrangulations associated to a  stratum $\strate$ allows one to compute the Masur-Veech volume of $\stratei$. In the case of translation surfaces, this approach were suggested by Kontsevich and Zorich and by Eskin and Okounkov.  In particular, it has been used in \cite{EO01} to compute the volumes of all strata of translation surfaces.

\subsection{Statements of the main results}\label{sec:state:results}
Our goal in this paper is to provide a method to compute the volumes of the moduli spaces of flat surfaces in genus $0$ where none of the angles at the singularities is in $2\pi\N$.
Let $\Lb_n$ denote the set of $(\mu_1,\dots,\mu_n) \in (\R_{<1})^n$ satisfying
\begin{equation}\label{eq:cond:weights}
\sum_{i=1}^n\mu_i=2.
\end{equation}
The parameter $\mu_i$ should be viewed as the curvature at the $i$-th marked point.
For all $\mu=(\mu_1,\dots,\mu_n) \in \Lb_n$, denote by $\CC(\mu)$ the space of  flat  surfaces of genus $0$, considered up to a rescaling, that have exactly $n$ singular points with prescribed cone angles $(\theta_1,\dots,\theta_n)$ at the singularities, where $\theta_i$ and $\mu_i$ are related by
\begin{equation}\label{eq:weight:n:cone:angle}
\theta_i=2\pi(1-\mu_i).
\end{equation}
Given  $(\CP^1,x_1,\dots,x_n) \in \Mod_{0,n}$, there is a unique flat metric on $\CP^1$ with cone angle $\theta_i$ at $x_i$. Therefore, one can identify $\CC(\mu)$ with $\Mod_{0,n}$. Thurston~\cite{Th98} and Veech~\cite{Veech:FS} have defined a natural volume form on $\CC(\mu)$. Since $\CC(\mu)$ can be identified with $\Mod_{0,n}$, we get a family of measures $d\vol_\mu$ on $\Mod_{0,n}$ parametrized by $\Lb_n$.

We first focus on the case where the flat surfaces in $\CC(\mu)$ are defined by $d$-differentials on the sphere. Let us fix a natural number $d\geq 2$. Given a vector $\kappa=(k_1,\dots,k_n)$, with $n\geq 3$,  such that $k_1+\dots+k_n=-2d$, and $k_i\geq 1-d$ for all $i=1,\dots,n$, we denote by $\stratesp$ the stratum of $d$-differentials associated to $\kappa$. Elements of $\stratesp$ are tuples $(\CP^1,x_1,\dots,x_n,q)$, where $\{x_1,\dots,x_n\}$ is a finite subset of $n$ distinct points in $\CP^1$, and $q$ is a meromorphic $d$-differential on $\CP^1$ such that
$\mathrm{div}(q)=k_1\cdot x_1+\dots+k_n\cdot x_n$. The projectivization of $\stratesp$ is  the quotient $\stratesp/\C^*$ and will be denoted by $\pstratesp$.
As explained above, $\pstratesp$ can be identified with the moduli spaces of flat surfaces $\CC(\mu)$, where $\mu_i:=-\frac{k_i}{d}, \; i=1,\dots,n$.
%
It was shown in \cite{Ng19} that every stratum of projectivized $d$-differentials carries a canonical volume form that will be denoted by $d\vol_1$ (see~\textsection~\ref{sec:loc:chart:vol:form} for more details). In genus $0$, this volume form agrees with $d\vol_\mu$.

\medskip

Our main result  in this context is Theorem~\ref{th:recursive:1} in which we prove a recursive formula relating the volumes of $\pstratesp$ to the volumes of other strata of lower dimensions in the case the vector $\kappa=(k_1,\dots,k_n)$ satisfies
\[
(*) \qquad  d \nmid k_i, \quad \text{ for all } \quad i=1,\dots,n.
\]
In the case $n=4$, epxlicit values of the volumes of $\stratesp$ can be computed (see Appendix~\ref{sec:low:dim:calcul}). Thus in principle, Theorem~\ref{th:recursive:1} allows one to compute the volumes of all strata $\pstratesp$ where $\kappa$ satisfy $(*)$. It is worth noticing that in \cite{Veech:FS}, Veech proved an integral formula that gives the volume of $\pstratesp$. However, to the author's knowledge, no explicit value of this integral was known. During the preparation of this paper, A. Sauvaget informed the author that he also obtained a different recursive formula computing the volumes of $\pstratesp$ where $\kappa$ satisfies $(*)$ \cite{Sau20, Sau21}.

\medskip

Since the formulation of Theorem~\ref{th:recursive:1} is rather long, we postpone it to \textsection~\ref{sec:recursive}. Making use of it, we  will prove a closed formula computing the volume of $\Omega^2\Mod_{0,n}(\kappa)$ when all the entries of $\kappa$ are odd numbers. Namely, we have
\begin{Theorem}\label{th:vol:quad:diff:stra}
For all $\kappa=(k_1,\dots,k_n) \in (\Z_{\geq -1})^n$ such that $k_1+\dots+k_n=-4$, we have
\begin{equation}\label{eq:vol:q:diff:g:0}
\vol_1(\Pb\Omega^2\Mod_{0,n}(\kappa))=\frac{(-1)^{\frac{n-2}{2}}}{2^{n-3}}\cdot\frac{\pi^{n-2}}{n-2}\cdot\prod_{i=1}^n\frac{k_i!!}{(k_i+1)!!}
\end{equation}
where $k!!=k\cdot(k-2)\cdots3\cdot1$ for $k$  odd, and $k!!=k\cdot(k-2)\cdots2$ for $k$ even, with the convention $0!!=(-1)!!=1$.
\end{Theorem}

Denote by $V_{MV}(k_1,\dots,k_n)$ the Masur-Veech volume of $\Pb\Omega^2\Mcal_{0,n}(k_1,\dots,k_n)$.
In \cite{AEZ16}, Athreya-Eskin-Zorich proved the following beautiful formula, which was conjectured by M. Kontsevich
\begin{equation}\label{eq:Kont:formula}
V_{MV}(k_1,\dots,k_n)=2\pi^2\cdot\prod_{i=1}^nv(k_i), \quad \text{ where }
\quad v(k):=\frac{k!!}{(k+1)!!}\cdot \pi^k\cdot\left\{
\begin{array}{ll}
\pi & \hbox{ for $k$ odd},\\
2   & \hbox{ for $k$ even}.
\end{array}
\right.
\end{equation}
An essential ingredient of the proof of \eqref{eq:Kont:formula} in \cite{AEZ16} is a result of Eskin-Kontsevich-Zorich~\cite{EKZ14} on the sum of the Lyapunov exponents of the Teichm\"uller geodesic flow. Athreya-Eskin-Zorich asked whether one can prove \eqref{eq:Kont:formula} without the result of \cite{EKZ14}.
In \textsection\ref{sec:compare:Kont:form}, we will show that  \eqref{eq:vol:q:diff:g:0} is equivalent to \eqref{eq:Kont:formula} thus giving an affirmative response to this challenge in the case all of the $k_i$'s are odd.
Unfortunately, the techniques in this paper do not allow one to compute the volume of $\Pb\Omega^2\Mcal_{0,n}(\kappa)$, where some of the $k_i$'s are even. We hope to return to this problem in a future work.

\medskip

We now turn to the space $\CC(\mu)$ where the entries of $\mu$ are not necessarily rational numbers. As another application of Theorem~\ref{th:recursive:1}, we will show

\begin{Theorem}\label{th:vol:funct}
There is a continuous piecewise polynomial function $\AA_n: \Lb_n \to \R$ such that
\begin{itemize}
\item[(a)] each polynomial piece of $\AA_n$ has degree at most $n-3$ and  rational coefficients,

\item[(b)] for all $\mu=(\mu_1,\dots,\mu_n) \in \Lb_n$ such that $\mu_i \notin \Z$, for all $i=1,\dots,n$, we have
\begin{equation}\label{eq:vol:funct}
\vol_{\mu}(\CC(\mu))=\frac{(-1)^{n-3}\pi^{n-2}}{(n-2)!}\cdot\AA_n(\mu)
\end{equation}

\item[(c)] if there exists $i\in \{1,\dots,n\}$ such that $\mu_i \in \Z$, then $\AA_n(\mu)=0$.
\end{itemize}
\end{Theorem}
\begin{Remark}\label{rk:vol:funct}\hfill
\begin{itemize}
\item[$\bullet$] In the case $0< \mu_i < 1$ for all $i=1,\dots,n$, different concrete expressions of $\vol_\mu(\CC(\mu))$ can be found in \cite[Th. 1.2]{McM_conegb} and \cite[Th. 1.1]{KN18}.

\item[$\bullet$] The author has been informed by A. Sauvaget that from the results of \cite{Sau20}, he has obtained a similar result to Theorem~\ref{th:vol:funct}, where on the right hand side of \eqref{eq:vol:funct} one has a piecewise polynomial functions with coefficients in $\Q[\pi]$.
\end{itemize}
\end{Remark}

\subsection{Organization} The paper is organized as follows: in \textsection\ref{sec:loc:chart:vol:form} we recall the period local charts and the definition of volume forms on moduli spaces of (projectivized) $d$-differentials. In \textsection\ref{sec:IVC:def} we introduce  the incidence variety compactification $\blowupsp$ of $\stratesp$.  We are particularly interested in the geometry of the irreducible components of the boundary of $\blowupsp$. In \textsection\ref{sec:tauto:ln:bdl:bdry}, we investigate the restriction of the tautological line bundle on $\blowupsp$ to its boundary.
Our goal is to determine a divisor representing the tautological line bundle on each boundary component.
\textsection\ref{sec:recursive} is devoted to the proof of the recursive formula  \eqref{eq:recursive:1}.
In \textsection\ref{sec:vol:strat:QD} we give the proof of Theorem \ref{th:vol:quad:diff:stra}. This proof makes use of a result on symmetric polynomials, which is proved in Appendix \ref{sec:sym:poly:append}.
The proof of Theorem~\ref{th:vol:funct} is given in \textsection \ref{sec:vol:mod:flat:surf}.

Finally, in Appendix \ref{sec:low:dim:calcul}, which is a joint work with Vincent Koziarz, we compute explicitly the  values of the Masur-Veech volumes of all strata $\Pb\stratesp$, where $n\in \{4,5\}$, $d\in \{2,3,4,6\}$, and none of the $k_i$'s is divisible by $d$.

\subsection*{Acknowledgements:}
the author is grateful to Yohan Brunebarbe  for the enlightening and inspiring discussions, which played an important role in the realization of this work. He warmly thanks to Pascal Autissier and Eric Balandraud for their helpful comments.
He also thanks  Martin M\"oller and Adrien Sauvaget for the stimulating discussions.

The author is partly supported by the French ANR project ANR-19-CE40-0003.


\section{Volume form on moduli space of $d$-differentials}\label{sec:loc:chart:vol:form}
In this section, we review some well known properties of the moduli spaces of $d$-differentials.
In what follows  $g$ will be a fixed non-negative integer and $\kappa:=(k_1,\dots,k_n)$ a vector of integers such that $k_i \geq 1-d$ and $k_1+\dots+k_n=d(2g-2)$.
\subsection{Canonical cyclic cover and local chart}\label{sec:cyclic:cov:loc:charts}
Recall that $\strate$ is the space of pairs $(X,q)$ where $X$ is a compact Riemann surface of genus $g$ and $q$ is a meromorphic $d$-differential on $X$ whose zeros and poles have orders prescribed by $\kappa$.  Here we always suppose that the zeros and poles of $q$ are numbered.
The datum of the marked points on $X$ is intentionally dropped  because they are implicitly included in the set of zeros and poles of $q$.
The quotient $\strate/\C^*$ is denoted by $\pstrate$.
The image of $(X,q)$ in $\pstrate$ is denoted by $(X,[q])$ and will be called a {\em projectivized $d$-differential}.
The spaces $\strate$ have been studied by different approaches.
In particular, we have
\begin{Theorem}[\cite{Veech:FS,BCGGM2,Sch18}]\label{thm:strata:dim}
Each stratum $\strate$ is an algebraic complex orbifold and
$$
\dim_\C\strate=\left\{
\begin{array}{cl}
2g+n-1 & \hbox{if elements of $\strate$ are $d$-th powers of Abelian differentials},\\
2g+n-2 & \hbox{otherwise.}
\end{array}
\right.
$$
\end{Theorem}

\begin{Remark}\label{rk:proj:strate:orbifold}
Theorem~\ref{thm:strata:dim} implies that $\Pb\strate$ is also a complex orbifold.
\end{Remark}

Let $\zeta$ be a fixed primitive $d$-th root of unity.
Let $(X,q)$ be a primitive $d$-differential in $\strate$.
The {\em canonical cyclic covering} of $(X,q)$ is a triple $(\hX,\homg,\tau)$, where $\hX$ is a compact Riemann surface, $\homg$ a holomorphic $1$-form on $\hX$, and $\tau$  an automorphism  of order $d$ of $\hX$ such that
\begin{itemize}
\item[(i)] $X$ is isomorphic to $\hX/\langle \tau \rangle$,

\item[(ii)] $\homg^d=\varpi^*q$, where $\varpi: \hX \to X\simeq \hX/\langle \tau \rangle$ is the canonical projection,

\item[(iii)] $\tau^*\homg=\zeta\homg$.
\end{itemize}
It is well known that such a triple always exists and is unique up to isomorphism (see for instance \cite{EV92,BCGGM2, Ng19}).

Let $\hZ$ denote the inverse image of the set of zeros and poles of $q$ in $\hX$.
Since $\tau$  preserves $\hZ$, $\tau$ induces an isomorphism $\tau^*$ of $H^1(\hX,\hZ,\C)$.
Let
$$
V_\zeta:=\ker(\tau^*-\zeta\id) \subset H^1(\hX,\hZ,\C).
$$
One can define a {\em period mapping} $\Phi: \tilde{U} \to V_\zeta$, where $\tilde{U}$ is a finite cover of a  neighborhood $U$  of $(X,q)$ in $\strate$, which is locally biholomorphic.
The period mappings form an atlas of orbifold local charts for $\strate$.
Since the transition functions of those charts are given by linear isomorphisms, this atlas endows  $\strate$ with a complex affine orbifold structure (see \cite{BCGGM2} or \cite{Ng19} for more details).
Note also that $\Phi$ induces a local chart  of $\Pb\strate$ which maps a neighborhood of $(X,[q])$ onto an open subset of the projective space $\Pb V_\zeta$.

\subsection{Volume form}\label{sec:vol:form:def}
Let $\pp: H^1(\hX,\hZ;\C) \ra H^1(\hX;\C)$ denote the canonical projection. For any $\eta \in  H^1(\hX,\hZ;\C)$, $\pp(\eta)$ is simply the restriction of $\eta$ to the  cycles in $H_1(\hX,\Z)$.
We have
\begin{equation*}
\dim(\ker\pp\cap V_\zeta)=\card\left(\{k_1,\dots,k_n\}\cap(d\Z)\right) \quad \text{ and } \quad \pp(V_\zeta)=\ker(\tau^* -\zeta\Id) \subset H^1(\hX,\C).
\end{equation*}
(for a proof see for instance \cite[Lem. 4.1, Prop.4.2]{Ng19}).

In \cite{Ng19}, it is shown that $\strate$  carries a canonical volume form $d\vol$ which is parallel with respect to its affine complex orbifold structure.
In the case $\ker\pp\cap V_\zeta=\{0\}$, the volume form $d\vol$ is defined as follows: on $H^1(\hX,\C)$ we have a Hermitian form $(.,.)$ arising from the intersection form on $H^1(\hX,\R)$.
It is not difficult to see that the restriction of $(.,.)$ to $\pp(V_\zeta)$ is non-degenerate.
By a slight abuse of notation, we also denote the pullback of this Hermitian form to $H^1(\hX,\hZ,\C)$ by $(.,.)$.
Since  $\pp_{|V_\zeta}$ is an isomorphism, the restriction of  $(.,.)$ to $V_\zeta$ is also non-degenerate.
Denote by $\vartheta$ the imaginary part of $(.,.)_{|V_\zeta}$.
Then $\frac{1}{N!}\vartheta^N$ is a volume form on $V_\zeta$, where $N=\dim V_\zeta=\dim \strate$.
Concretely, if $(z_1,\dots,z_N)$ is a system of coordinates on $V_\zeta$ in which $(.,.)_{|V_\zeta}$ is given by a Hermitian matrix $H$, then
\begin{equation}\label{eq:local:form:vol}
\frac{\vartheta^N}{N!}=\det(H)\cdot \left(\frac{\imath}{2}\right)^N dz_1\wedge d\bar{z}_1\wedge\dots\wedge dz_N\wedge d\bar{z}_N.
\end{equation}
In particular, if $(.,.)_{|V_\zeta}$ is given by the matrix $\mathrm{diag}(\underset{r}{\underbrace{1,\dots,1}},\underset{s}{\underbrace{-1,\dots,-1}})$, where $r+s=N$, then
$$
\frac{1}{N!}\vartheta^N=(-1)^s\left(\frac{\imath}{2}\right)^N dz_1\wedge d\bar{z}_1\wedge\dots\wedge dz_N\wedge d\bar{z}_N.
$$
Using local charts by period mappings, this volume form gives rise to a  well defined volume form $d\vol$ on $\strate$.

Let $V_\zeta^+$ denote the  cone $\{v \in V_\zeta, \; (v,v)>0\}$. The projection of $V^+_\zeta$ in $\Pb V_\zeta$ will be denoted by $\Pb v^+_\zeta$.
The volume form $d\vol$ on $V_\zeta$ induces a volume form $d\vol_1$ on $V^+_\zeta$ as follows: given  be any open subset $B$ of $\Pb V^+_\zeta$, let $C_1(B)=\{v \in V^+_\zeta, \; 0 < (v,v) < 1, \; \langle v \rangle \in B \}$. Then by definition $\vol_1(B):=\vol(C_1(B))$.

By construction, if $v=\Phi(X,q)$ then $(v,v)=\Aa(\hX,\homg)=d\cdot\Aa(X,q)>0$, which means that  the image of the period mapping $\Phi$ is contained in $V^+_\zeta$. Taking quotient by $\C^*$, we see that $\pstrate$ locally modeled by $\Pb V^+_\zeta$. In particular, the volume form $d\vol_1$ induces a well-defined volume form on $\pstrate$.
In \cite{Ng19}, it is shown that the total volume of $\pstrate$ with respect to $d\vol_1$ is finite.

\medskip

Let $\istrate$ denote the set of $(X,q) \in \strate$ whose area  at most $1/d$, or equivalently such that the area of the canonical cyclic cover $(\hX,\homg)$ is at most $1$. Remark that for each $q$ we have $d$  choices for $\homg$ (any two choices differs by a $d$-th root of unity). Therefore, the volumes of $\istrate$ and of $\pstrate$ are related by the following
\begin{equation}\label{eq:rel:vols:proj:n:area:1}
\vol(\istrate)=\frac{1}{d}\cdot\vol_1(\pstrate)
\end{equation}

\subsection{Comparison with Masur-Veech volume}\label{sec:vol:form:compare:MV}
 If $d\in \{1,2,3,4,6\}$, there is another natural volume form on $\strate$  that we now  describe.
 Recall that if $A$ is a $\Z$-module, then $H^1(\hX,\hZ,A)$ is the space of morphisms of $\Z$-modules $\eta: H_1(\hX,\hZ,\Z) \to A$.
 Let
 $$
 \Lambda_\zeta= \left\{
 \begin{array}{ll}
  V_\zeta\cap H^1(\hX,\hZ,\Z[\imath]) & \text{ if } d\in \{2,4\},\\
  V_\zeta\cap H^1(\hX,\hZ,\Z[e^{\frac{2\pi\imath}{3}}]) & \text{ if } d\in \{3,6\}.
 \end{array}
 \right.
 $$
 Then $\Lambda_\zeta$ is a lattice of $V_\zeta$. There is unique volume form on $V_\zeta$ such that the co-volume of $\Lambda_\zeta$ is $1$.
 Since the transition maps of the local charts by the period mappings preserves $\Lambda_\zeta$, this  volume form is well defined on $\strate$. It  will be referred to as the {\em Masur-Veech measure} and denoted by $d\vol^{MV}$.
 This volume form induces a volume form $d\vol_1^{MV}$ on $\Pb\strate$ in the same manner as $d\vol_1$ is induced from $d\vol$.
 In fact, the Masur-Veech volume arises naturally from the counting problem of tilings of surfaces by triangles and squares  (see for instance \cite{EO01,ES18,Engel-I,KN20}).
 From the definition, there is a constant $\lambda \in \R^*$  such that $d\vol=\lambda d\vol^{MV}$ and  $d\vol_1=\lambda d\vol^{MV}_1$.
 Note that the constant $\lambda$ can be negative. It can be shown that $\lambda$ always belongs to either $\Q$ or $\sqrt{3}\Q$ (see \cite[Prop.5.9]{Ng19}).


\section{The incidence variety compactification}\label{sec:IVC:def}
In order to apply intersection theory to compute the volume, one first needs to find an adequate compactification of $\pstratesp$. One such compactification is the incidence variety which is defined as follows (see~\cite{Gen18, BCGGM1, BCGGM2, Ng23:a}): there is a holomorphic vector bundle $\khodgespext$ over $\ol{\Mod}_{0,n}$ of rank $(d-1)(n-2)-1$ whose fiber over a point $\xx\simeq (C,x_1,\dots,x_n) \in \ol{\Mod}_{0,n}$  is identified with $H^0(C,d\cdot\omega_C+\sum_{i=1}^n(d-1)\cdot x_i)$, where $\omega_C$ is the dualizing sheaf of $C$. When $C$ is smooth, the fiber of $\khodgespext$ over $\xx$ is the space of $d$-differentials on $C$ whose poles are contained in the set $\{x_1,\dots,x_n\}$, and the orders of the poles are at most $d-1$. By definition, $\stratesp$ is a subset of the total space of $\khodgespext$, and $\pstratesp$ is a subvariety of the projective bundle $\pkhodgespext$ associated with $\khodgespext$. The incidence variety compactification $\clpstratesp$ of $\pstratesp$ is defined to be its closure  in $\pkhodgespext$.

Let $\OO(-1)_{\pkhodgespext}$ denote the tautological line bundle over $\pkhodgespext$. It follows from a result by Costantini-M\"oller-Zachhuber~\cite{CMZ19} that in the case where none of the $k_i$'s is divisible by $d$, the volume of $\pstratesp$ can be computed from the intersection number $c_1^{n-3}(\OO(-1)_{\pkhodgespext})\cap [\clpstratesp]$, where  $[\clpstratesp]$ is the rational equivalence class of $\clpstratesp$ in $A_*(\pkhodgespext)$ (recall that $\dim\pstratesp=\dim \Mod_{0,n}=n-3$).
In particular, for the volume form $d\vol_1$ defined in \textsection~\ref{sec:vol:form:def}, we have
\begin{Theorem}[\cite{CMZ19, Ng22}]\label{th:vol:n:inters:nbr}
Assume that $d \nmid k_i$, for all $i=1,\dots,n$. Then
\begin{equation}\label{eq:vol:n:inters:nbr:g0}
\vol_1(\pstratesp)=\frac{(-1)^{n-3}}{d^{n-3}}\cdot\frac{(2\pi)^{n-2}}{2^{n-2}(n-2)!}\cdot c^{n-3}_1(\OO(-1)_{\pkhodgespext})\cdot[\clpstratesp].
\end{equation}
\end{Theorem}

The incidence variety compactifications of strata of $d$-differentials (in any genus) have been the object of several works \cite{Gen18, BCGGM1, BCGGM2, Sau20}. In genus $0$, it has been shown in \cite{Ng23:a} that $\clpstratesp$ is isomorphic to the blowup $\hat{p}:\blowupsp \to \ol{\Mod}_{0,n}$ of $\ol{\Mod}_{0,n}$ along an explicit sheaf of ideals $\Ic$.

Recall that $\partial\ol{\Mod}_{0,n}$ is a simple normal crossing divisor, whose components are in bijection with the set $\Pc$ of partitions $\Sc=\{I_0,I_1\}$ of $\{1,\dots,n\}$ such that $\min\{|I_0|, |I_1|\}\geq 2$.
Let $\partial\blowupsp:=\hat{p}^{-1}(\partial\ol{\Mod}_{0,n})$, where $\partial\ol{\Mod}_{0,n}:=\ol{\Mod}_{0,n}\setminus\Mod_{0,n}$ is the boundary of $\ol{\Mod}_{0,n}$. We will call $\partial \blowupsp$ the boundary of $\blowupsp$.

Define $\mu_i:=-\frac{k_i}{d}, \; i=1,\dots,n$, $\mu=(\mu_1,\dots,\mu_n)$, and for all $I\subset\{1,\dots,n\}$, $\mu(I):=\sum_{i\in I}\mu_i$.
The following theorem is proved in \cite[\textsection 4]{Ng23:a}
\begin{Theorem}\label{th:boundary:comp}
Each irreducible component of $\partial\blowupsp$ is a divisor.
The set of irreducible components of $\partial \blowupsp$ is in bijection with the set $\hat{\Pc}(\mu)$ of partitions $\Scal=\{I_0,I_1,\dots,I_r\}$ of $\{1,\dots,n\}$ with $r\geq 1$ such that either

\begin{itemize}
\item[(a)] $r=1$ and $\Sc=\{I_0,I_1\} \in \Pcal$, or

\item[(b)] $r\geq 2$ and $\Scal$ satisfies $\mu(I_0)<1$, and $\mu(I_j)>1$ for all $j=1,\dots,r$.
\end{itemize}
More precisely,  every partition $\Sc$ in  $\hat{\Pc}(\mu)$ determines  a unique stratum $D^*_\Sc$ in $\ol{\Mod}_{0,n}$ as follows: any pointed curve $(C,x_1,\dots,x_n)$ parametrized by a point in $D^*_{\Sc}$ has $r+1$ irreducible components denoted by $C^0,\dots,C^r$, where $C^j$ contains the marked points $\{x_i, _; i\in I_j\}$, and there is a node between $C^0$ and $C^j$ for all $j=1,\dots,r$. Then $\hat{D}_\Sc:=\ol{\hat{p}^{-1}(D^*_\Sc)}$ is
an irreducible component of $\partial \blowupsp$. Conversely, every irreducible component of $\partial\blowupsp$ arises this way.
\end{Theorem}

Here is a more detailed description of the component $\hat{D}_\Sc$ associated with a partition $\Sc=\{I_0,\dots,I_r\} \in \hat{\Pc}(\mu)$. Let $\xx \simeq (C,x_1,\dots,x_n)$ be a point in the stratum $D^*_\Sc$ and $C^0,\dots,C^r$ the irreducible components of the curve $C$ as in Theorem~\ref{th:boundary:comp}. A neighborhood of $\xx$ in $\ol{\Mod}_{0,n}$ can be identified with $\Uc_\xx=\Delta^r\times U$, where $\Delta \subset \C$ is a small disc about $0$, and $U$ is an open subset of $\C^{n-r-3}$. Note that $U$ is identified with a neighborhood of $\xx$ in $D^*_\Sc$. Let $u$ be a holomorphic system of coordinates on $U$, and  $t=(t_1,\dots,t_r)$ be the coordinates on $\Delta^r$, where $t_j$ is the parameter associated with the node between $C_0$ and $C^j$. Let
$m_j:=d(\mu(I_j)-1)$, where $\mu(I_j)=\sum_{i\in I_j}\mu_i$. Define
$$
t^{\beta_0}:=\prod_{j=1}^rt_j^{m_j}, \quad \text{ and } \quad t^{\beta_j}:=\frac{t^{\beta_0}}{t_j^{m_j}}.
$$
Let $[v]:=[v_1:\dots:v_r]$ be the homogeneous coordinates on the projective space $\CP^{r-1}$. Then we have
\begin{equation}\label{eq:loc:eqn:bdry:div}
\widehat{\Uc}_\xx:=\hat{p}^{-1}(\Uc_\xx)\simeq \{(t,u,[v]) \in \Uc_\xx\times\CP^{r-1}, \; v_jt^{\beta_k}=v_kt^{\beta_j}, j,k \in \{1,\dots,r\}\} \subset \Uc_\xx\times\CP^{r-1}.
\end{equation}
In this setting, we have $\hat{D}_\Sc\cap\widehat{\Uc}_\xx$ is defined by the equations $t_1=\dots=t_r=0$, and therefore $\hat{D}_\Sc\cap \widehat{\Uc}_\xx$ is isomorphic to $U\times \CP^{r-1}$.

\medskip

Let $\hat{\Lc}_\mu$ denote the restriction of the tautological line bundle $\OO(-1)_{\pkhodgespext}$ to $\blowupsp$. To apply Theorem~\ref{th:vol:n:inters:nbr}, it is important to have an explicit a divisor representative of $\hat{\Lc}_\mu$. This is achieved in
\begin{Theorem}[\cite{Ng23:a}]\label{th:tauto:div}\hfill
\begin{itemize}
\item[(a)] For all $\Sc=\{I_0,I_1,\dots,I_r\} \in \hat{\Pc}(\mu)$,  define $m(\Sc):=d^r\cdot \prod_{j=1}^r(\mu(I_j)-1)$.
Then the  exceptional divisor $\Ec$ of the blowup $\hat{p}: \blowupsp \to \ol{\Mod}_{0,n}$ satisfies
\begin{equation}\label{eq:except:Weil:div}
\Ec \sim \sum_{\Sc \in \hat{\Pc}(\mu)}(|\Sc|-2)\cdot m(\Sc)\cdot\hat{D}_\Sc
\end{equation}
where $|\Sc|$ is the length of $\Sc$.

\item[(b)] Let $\Dc_\mu$ denote the following $\Q$-divisor in $\ol{\Mod}_{0,n}$
\begin{equation}\label{eq:express:D:mu}
\Dc_\mu := \frac{d}{(n-2)(n-1)}\sum_{\Scal=\{I_0,I_1\}\in \Pcal}(|I_0|-1)(|I_1|-1-(n-1)\mu_\Scal)\cdot D_\Scal.  
\end{equation}
where $\mu_\Sc=\mu(I_1)-1$, for all $\Sc=\{I_0,I_1\} \in \Pc$.
Then we have
\begin{equation}\label{eq:tauto:ln:bdl:expr}
\hat{\Lc}_\mu \sim \hat{\Dc}_\mu:= \hat{p}^*\Dc_\mu+\Ecal.
\end{equation}
\end{itemize}
\end{Theorem}

\section{Tautological line bundle on the boundary of $\blowupsp$}\label{sec:tauto:ln:bdl:bdry}
In this section, we investigate the restriction of the tautological line bundle $\hat{\Lcal}_\mu$ to irreducible components of  $\partial\blowupsp$.
The results in this section will allow us to derive the recursive formula in Theorem~\ref{th:recursive:1}.

By Theorem~\ref{th:boundary:comp}, we know that each irreducible component of $\partial\blowupsp$ corresponds uniquely to a partition $\Sc$ in $\hat{\Pc}(\mu)$.
We divide $\hat{\Pcal}(\mu)$ into two subsets: $\hat{\Pcal}_0(\mu)$ is the set of partitions $\Sc=\{I_0,I_1\}\in \Pcal$ such that $\mu(I_0)=\mu(I_1)=1$, and $\hat{\Pcal}_1(\mu)=\hat{\Pcal}(\mu)\setminus\hat{\Pcal}_0(\mu)$.
Throughout this section, given $\Sc\in \hat{\Pc}(\mu)$, we denote by $D^*_\Scal$ the stratum of $\ol{\Mod}_{0,n}$ associated to $\Scal$.
Note that we have $\hat{D}_\Sc=\ol{\hat{p}^{-1}(D^*_\Sc)}$ (cf. Theorem~\ref{th:boundary:comp}).

\subsection{Tautological line bundle on boundary divisors in $\hat{\Pcal}_1(\mu)$}\label{subsec:tauto:l:bdl:on:div:P1}
Consider a partition $\Scal=\{I_0,\dots,I_r\} \in \hat{\Pcal}_1(\mu)$.
Let $n_j=|I_j|, \; j=0,\dots,r$.
Recall that the stable curves parametrized by  $D^*_\Sc$ have $r+1$ components each of which corresponds to a subset $I_j$ in the partition $\Scal$.
Since the component associated to $I_0$ is adjacent to all the other components, we have
$$
D^*_\Sc \simeq \Mod_{0,n_0+r}\times\Mod_{0,n_1+1}\times\dots\times\Mod_{0,n_r+1}.
$$
For each $j=1,\dots,r$, we define a vector $\kappa^{(j)}\in \Z^{n_j+1}$ by setting the first coordinate of $\kappa^{(j)}$ to be
$$
k^{(j)}_{1}:=-2d-\sum_{i\in I_j}k_i=\sum_{i\not\in I_j}k_i
$$
and  the remaining coordinates of $\kappa^{(j)}$ are $\{k_i, \;i \in I_j\}$ (the ordering of $\{k_i, \;i \in I_j\}$ does not matter).
By definition, we have  $k^{(j)}_{1}>-d$.
Consider the stratum of $d$-differentials $\Omega^d\Mod_{0,n_j+1}(\kappa^{(j)})$.
Let $\widehat{\Mod}_{0,n_j+1}(\kappa^{(j)})$ be the incidence variety compactification of $\Pb\Omega^d\Mod_{0,n_j+1}(\kappa^{(j)})$.
Define
$$
\Mod_\Scal:=\ol{\Mod}_{0,n_0+r}\times\left(\prod_{j=1}^r\widehat{\Mod}_{0,n_j+1}(\kappa^{(j)})\right).
$$
For $j=1,\dots,r$, let $p_{\Scal,j}: \Mod_\Scal \to \widehat{\Mod}_{0,n_j+1}(\kappa^{(j)})$ be the canonical projection.
Let $\hat{\Lc}_j$ denote the tautological line bundle on  $\widehat{\Mod}_{0,n_j+1}(\kappa^{(j)})$.
For simplicity, we will denote by $\OO(-1)$ the tautological line bundle on $\Pb\ol{\Hcal}^{(d)}_{0,n}$.

\begin{Proposition}\label{prop:bdry:div:bl:P1}
  Let $E_\Scal$ denote the vector bundle $ p_{\Scal,1}^*\hat{\Lcal}_1\oplus\dots\oplus p_{\Scal,r}^*\hat{\Lcal}_r$ over $\Mcal_\Scal$.
  Then there is a birational surjective morphism $i_\Scal: \Pb E_\Scal \to \hat{D}_\Scal$ such that $i_\Scal^*\OO(-1)\sim \OO(-1)_{\Pb E_\Scal}$.
\end{Proposition}
\begin{proof}
Locally on $\widehat{\Mcal}_{0,n_j+1}(\kappa^{(j)})$, a trivializing section of $\hat{\Lc}_j$ is given by a family of pairs $(C^j, q_j)$, where $C^j$ is a stable curve in $\ol{\Mod}_{0,n_j+1}$ and $q_j$ a non-zero $d$-differential on $C^j$.
Therefore, locally an element of $E_\Scal$ can be identified with a tuple $((C^0,0),(C^1,q_1),\dots,(C^r,q_r),z)$, where
\begin{itemize}
\item[$\bullet$] $(C^0,0)$  is a stable curve parametrized by  $\ol{\Mod}_{0,n_0+r}$ equipped with the zero $d$-differential,

\item[$\bullet$]$(C^j,q_j)$ is a stable curve $C^j$ parametrized by $\ol{\Mod}_{0,n_j+1}$ equipped with a $d$-differential $q_j \not\equiv 0$ such that   $(C^j,q_j)\in \Omega^d\ol{\Mod}_{0,n_j+1}(\kappa^{(j)})$ for $j=1,\dots,r$,

\item[$\bullet$] $z=(z_1,\dots,z_r)\in \C^r$.
\end{itemize}
Such a tuple defines an element of $\khodgespext$ as follows: let $C$ be the stable curve represented by a point in $S$, which is obtained by identifying the first marked point on $C^j$ with a marked point on $C^0$ to form a node for $j=1,\dots,r$.
We consider the $C^j$'s as subcurves of $C$. Then the family $\{z_1q_1,\dots,z_rq_r\}$ gives a $d$-differential $q$ on $C$ ($q$ vanishes identically on $C^0$). The pair $(C,q)$ is clearly an element of $\khodgespext$.
Thus, by taking the projectivization, we get a map $i_\Scal : \Pb E_\Scal \to \pkhodgespext$.

Let $\Pb E^*_\Sc$ be the restriction of $\Pb E_\Sc$ to $D^*_\Sc=\Mod_{0,n_0+r}\times\Mod_{0,n_1+1}\times\dots\times\Mod_{0,n_r+1}\subset \Mcal_\Sc$.
If $C^j$ is a curve parametrized by a point in $\Mod_{0,n_j+1}$, then up to a multiplication by $\C^*$, there is a unique $d$-differential $q_j$ on $C^j$ such that $(C^j,q_j)\in \Omega^d \Mod_{0,n_j+1}(\kappa^{(j)})$.
Thus the restriction of $i_\Sc$ to $\Pb E^*_\Sc$ is an isomorphism between $\Pb E^*_\Scal$ and $\hat{p}^{-1}(D^*_\Sc) \subset \blowupsp \subset \pkhodgespext$.
It follows that $i_\Sc(\Pb E_\Sc)=\ol{\hat{p}^{-1}(D^*_\Sc)}=\hat{D}_\Sc$, and $i_\Sc: \Pb E_\Sc \to \hat{D}_\Sc$ is birational.
It is also clear from the definition that $i_\Sc^*\OO(-1)\sim \OO(-1)_{\Pb E_\Sc}$.
\end{proof}

\begin{Remark}\label{rk:div:bl:P1:bir:map}
The map $i_\Scal$ is not one-to-one in general. For instance, let $q_1$ and $q'_1$ be two $d$-differentials on the same curve $C^1$ such that both $(C^1,q_1)$ and $(C^1,q'_1)$ are elements of $\Omega^d\ol{\Mod}_{0,n_1+1}(\kappa^{(1)})$, but $q_1$ and $q'_1$ are not proportional (such $d$-differentials exist only if $C^1$ is reducible). Then the tuples $\allowbreak ((C^0,0),(C^1,q_1), \dots, (C^r,q_r), (0,\dots,0,1))$ and $((C^0,0),(C^1,q'_1),\dots,(C^r,q_r),(0,\dots,0,1))$ give the same element of $\Hcal^{(d)}_{0,n}$, that is a $d$-differential on $C$ which is equal to $q_r$ on $C^r$ and vanishes identically on the other subcurves.
\end{Remark}

\subsection{Tautological line bundle on boundary divisors in $\hat{\Pcal}_0(\mu)$}\label{subsec:tauto:l:bdl:div:P0}
Consider now a divisor $\hat{D}_\Sc$ in $\blowupsp$ with $\Sc=\{I_0,I_1\} \in \hat{\Pcal}_0(\mu)$.
In this case $\hat{D}_\Sc$ is the proper transform of the boundary divisor $D_\Sc$ in $\partial\ol{\Mod}_{0,n}$ associated to $\Sc$.
Since $D_\Sc$ is not contained in the support of $\Ical$, $\hat{D}_\Sc$ is actually the blow-up of $D_\Sc$ along the ideal sheaf $\Ical_{|D_\Sc}$ (see \cite[Chap.II, Cor.7.15]{Hart}). Note also that the local generating sets of the ideal sheaf $\Ic$ do not involve the defining function of $D_\Scal$.

\subsubsection{Boundary of $\hat{D}_{\Sc}$}\label{sec:bdry:of:div:P0}
Recall that for all $\Sc' \in \hat{\Pc}(\mu)$, $D^*_{\Sc'}$ is the stratum of $\ol{\Mod}_{0,n}$ associated with $\Sc'$ and $D_{\Sc'}$ is the closure of $D^*_{\Sc'}$ in $\ol{\Mod}_{0,n}$.
\begin{Definition}\label{def:part:div:in:div:P0}
Let $\hat{\Pcal}(\mu,\Sc)$ denote the set of partitions $\Sc'$ in $\hat{\Pcal}(\mu)\setminus \{\Sc\}$ such that $D_{\Sc'}$  intersects $D_\Scal$.
For all $\Sc'\in \hat{\Pc}(\mu,\Sc)$, we denote  the intersection $D_{\Sc'}\cap D_\Sc$ by $D_{\Sc'\cdot\Sc}$.
The stratum of $\ol{\Mod}_{0,n}$ whose closure is equal to $D_{\Sc'\cdot\Sc}$ will be denoted by $D^*_{\Sc'\cdot\Sc}$.
\end{Definition}

By definition, $D_\Sc$ is the closure of the stratum $D^*_\Sc \simeq \Mod_{0,n_0+1}\times \Mod_{0,n_1+1}$ in $\ol{\Mod}_{0,n}$, where $n_j=|I_j|,  \; j=0,1$. One can consider $D^*_\Sc$ as an open dense subset in $\hat{D}_\Sc$. Denote by $\partial\hat{D}_\Sc$ the complement of $D^*_\Sc$ in $\hat{D}_\Sc$. Since the defining function of $D_\Sc$ in $\ol{\Mod}_{0,n}$ is not involved in the ideal sheaf $\Ic$, the same arguments as in the proof of \cite[Th. 4.1]{Ng23:a} yield
\begin{Proposition}\label{prop:bdry:div:bl:P0}
Each irreducible component of $\partial\hat{D}_\Scal$ is a divisor in $\hat{D}_\Scal$,
which corresponds to the intersection $\hat{D}_\Scal \cap \hat{D}_{\Scal'}$, where $\Scal'$ is an element of $\hat{\Pcal}(\mu,\Sc\}$.
Moreover, let $D^*_{\Sc'\cdot \Sc}$ be the stratum of $\ol{\Mod}_{0,n}$ such that $D_{\Sc'}\cap D_{\Sc}=\ol{D^*_{\Sc'\cdot\Sc}}$. Then we have
$$
\hat{D}_{\Sc'\cdot \Sc}:=\hat{D}_{\Sc'}\cap\hat{D}_\Sc=\ol{\hat{p}^{-1}(D^*_{\Scal'\cdot \Scal})}.
$$
\end{Proposition}
Our goal now is to prove the following

\begin{Theorem}\label{th:tauto:l:bdl:on:div:P0}
Let $\hat{\Pc}^*(\mu,\Sc)$ be the set of $\Sc'=\{I'_0,\dots,I'_r\}\in \hat{\Pc}(\mu,\Sc)$, such that $\Sc'\in \hat{\Pc}_1(\mu)$ and the curves parametrized by $D^*_{\Sc'\cdot\Sc}$ are obtained from the ones parametrized by $D^*_{\Sc'}$ by pinching a loop in the component associated to $I'_0$.
Then the restriction of the tautological line bundle $\OO(-1)$ to $\hat{D}_\Scal$ is represented by the Weil divisor
$$
W_\Scal:=-\sum_{\Scal' \in \hat{\Pcal}^*(\mu,\Scal)}m({\Scal'})\cdot \hat{D}_{\Scal'\cdot\Scal}.
$$
\end{Theorem}

\subsubsection{Meromorphic section of $\OO(-1)$ over $\hat{D}_\Sc$}\label{sec:sect:tauto:ln:bdl:div:P0}
To prove Theorem~\ref{th:tauto:l:bdl:on:div:P0}, we single out a meromorphic section of $\OO(-1)_{|\hat{D}_\Sc}$ which does not vanish in $D^*_\Sc$ and  compute its vanishing orders at the components of $\partial\hat{D}_\Sc$.
Let $q$ be a meromorphic $d$-differential on $\CP^1$ which has a pole of order $d$ at $x_0 \in \CP^1$.
The residue of $q$ at $x_0$ is defined as follows: let $(\hat{C}, \homg, \tau)$ be the canonical cyclic cover of $(\CP^1,q)$.
Let $\hat{x}_0$ be a point in $\hat{C}$ which is mapped to $x_0$.
The preimage of $x_0$ in $\hat{C}$ consists of $\{\tau^k(\hat{x}_0), \; k=0,\dots,d-1\}$.
The Abelian differential $\homg$ has a simple pole at $\hat{x}_0$.
Let $c$ be the residue of $\homg$ at $\hat{x}_0$.
Then the residue of $\homg$ at any other point in the preimage of $x_0$ is equal to $\lambda\cdot c$, where $\lambda$ is a $d$-th root of unity. Thus $c^d$ does not depend on the choice of $\hat{x}_0$, we will call it the residue of $q$ at $x_0$.
Concretely, if $z_0=0$, and $q=\frac{f(z)(dz)^d}{z^d}$, where $f(z)$ is a holomorphic function in a neighborhood of $0$ such that $f(0)\neq 0$, then the residue of $q$ at $0$ is simply $f(0)$.

We construct a trivializing section of $\OO(-1)_{|D^*_\Sc}$ as follows: for any $\xx\in D^*_\Sc$,  the curve $C_\xx$ has two components $C^0_\xx$ and $C^1_\xx$. For each $j\in \{0,1\}$, up to a scalar in  $\C^*$ there is a unique meromorphic $d$-differential $q_j(\xx)$ which vanishes to the order $k_i$ at the $i$-th marked point in $C^j_\xx$ (with $i\in I_j$) and to the order $-d$ at the unique node of $C_\xx$.
By construction $q_0(\xx)$ and $q_1(\xx)$ have  non-zero residues at the node of $C8\xx$. By multiplying $q_0(\xx)$ and $q_1(\xx)$ by a non-vanishing regular  function, we can assume that their residues at the node are both equal to $1$. This normalization allows us to patch the local sections $(q_0(\xx),q_1(\xx))$ together to obtain a trivializing section $\Phi$ of $\OO(-1)$ over $D^*_\Sc$.
The section $\Phi$ extends naturally to a meromorphic section of $\OO(-1)$ over $\hat{D}_\Sc$. 
We will compute the vanishing order of $\Phi$ along each of the irreducible components of $\partial\hat{D}_\Sc$.

\medskip

Let $\hat{\xx}$ be a point in $\partial\hat{D}_\Scal$, and $\xx$ its image in $D_\Sc$. We  suppose that $\hat{\xx}$ is a generic point in $\hat{D}_{\Scal'\cdot\Scal}$ which means that $\xx\in D^*_{\Sc'\cdot\Sc}$, for some $\Scal'=\{I'_0,\dots,I'_r\}\in \hat{\Pcal}(\mu,\Scal)$.
Note that $\mathrm{codim}D^*_{\Sc'\cdot \Sc}=\mathrm{codim}D^*_{\Sc'}+1=r+1$. Therefore, the pointed curve $(C_\xx,x_1,\dots,x_n)$ parametrized by $\xx$ has $r+1$ nodes and $r+2$ irreducible components. 
By definition, $C_\xx$ has a node $\alpha$ such that the two subcurves $\hat{C}^0_{\xx,\alpha}, \hat{C}^1_{\xx,\alpha}$ separated by $\alpha$ satisfy $x_i \in \hat{C}^j_{\xx,\alpha}$ if and only if $i\in I_j, \; j=0,1$.
Smoothening the node $\alpha$ of $C_\xx$, we obtain a pointed curve $C_{\xx'}$ parametrized by a point $\xx' \in D^*_{\Sc'}$.
Denote the irreducible components of $C_{\xx'}$ by $C^0_{\xx'},\dots,C^r_{\xx'}$, where $C^j_{\xx'}$  contains the marked points with index in $I'_j$.
Recall that for each $j\in \{1,\dots,r\}$,  there is a node between $C^0_{\xx'}$ and $C^j_{\xx'}$.
Topologically, $C_\xx$ is obtained from $C_{\xx'}$ by pinching a simple closed curve on an irreducible component $C^i_{\xx'}$.
Let us choose a numbering of the irreducible components of $C_\xx$ as follows: for $j\in\{0,\dots,r\}$ and $j\neq i$, $C^j_\xx$ is the component of $C_\xx$ which corresponds to $C^j_{\xx'}$.
The component $C^i_{\xx'}$ corresponds to the union of two components of $C_\xx$. We denote one of the two components by $C^i_\xx$ and the other one by $C^{r+1}_\xx$.

\medskip

We can identify a neighborhood of $\xx$ in $\ol{\Mod}_{0,n}$ with $\Uc_\xx=\Delta^{r+1}\times U$ where $\Delta\subset\C$ is a small disc about $0$, and $U$ is an open subset of $\C^{n-4-r}$. Let $t=(t_1,\dots,t_{r+1})$ be the coordinate functions on $\Delta^{r+1}$, where $t_j$ is the coordinate associated to the node between $C^0_{\xx'}$ and $C^j_{\xx'}$ for $j=1,\dots,r$ (here we identify a node of $C_{\xx'}$ with a node of $C_\xx$), and $t_{r+1}$ corresponds to the node $\alpha$.
By definition  $D_{\Scal'}\cap\Ucal_\xx$ is defined by the equations $\{t_j=0, \; j=1,\dots,r\}$ and $D_\Sc\cap \Uc_\xx$ is defined by $t_{r+1}=0$.

For $j=1,\dots,r$, set $m_j:=d(\mu(I'_j)-1)\in \Z_{>0}$. Define
$$
t^{\beta_0}=t_1^{m_1}\cdots t_r^{m_r}, \quad t^{\beta_k}=\frac{t^{\beta_0}}{t_k^{m_k}}, \quad k\in \{1,\dots,r\}, \quad \text{ and } t^{\beta_{r+1}}=t^{\beta_i}.
$$
Then $\widehat{\Uc}_\xx=\hat{p}^{-1}(\Uc_\xx)$ is defined by the equations $v_jt^{\beta_k}=v_kt^{\beta_j}, \; \forall j,k \in \{1,\dots,r\}$
in $\Uc_\xx\times\Pb^{r-1}_\C$. Observe that $\widehat{\Uc}_\xx$ is covered by the open subsets $\widehat{\Uc}^k_\xx=\{(t,u,[v_1:\dots:v_r]) \in \widehat{\Uc}_\xx, \; v_k\neq 0\}$, for $k=1,\dots,r$.
%
\begin{Lemma}\label{lm:div:P0:sigma:order:zero}
If $i>0$  then the order of $\Phi$ along $\hat{D}_{\Sc'\cdot\Sc}$ is zero.
\end{Lemma}
\begin{proof}
Without loss of generality, we can suppose that $i=r$.
We can choose the numbering of $C^0_\xx,\dots,C^{r+1}_\xx$ such that $C^r_\xx$ is adjacent to both $C^0_\xx$ and $C^{r+1}_\xx$. Recall that $\Sc=\{I_0,I_1\}$ is the partition associated with the node between $C^r_\xx$ and $C^{r+1}_\xx$.

In \cite[\textsection 2]{Ng23:a}, we have constructed the meromorphic sections $\Phi_0,\dots,\Phi_{r+1}$ of the bundle $K^{\otimes d}_{\ol{\Cc}_{0,n}/\ol{\Mod}_{0,n}}$ over $\ol{\Cc}_{0,n|\Uc_\xx}$. It was shown in \cite[\textsection 5]{Ng23:a} that $\Phi_k$ gives a trivializing section of the line bundle $\hat{\Lc}_\mu$ on the open subset $\widehat{\Uc}_\xx^k$, for all $k=1,\dots,r$. Since the node between $C^r_\xx$ and $C^{r+1}_\xx$ has weight zero, the ratio $\Phi_{r+1}/\Phi_r$ is a non-vanishing holomorphic function on $\Uc_\xx$ (hence on $\widehat{\Uc}_\xx$). In particular,  $\Phi_{r+1}$ gives a trivializing section of $\hat{\Lc}_\mu$ over $\widehat{\Uc}^r_{\xx}$.

Consider now a point $\yy \in D^*_\Sc\cap \Uc_\xx$.  The curve $C_\yy$ has two components $C^0_\yy$ and $C^1_\yy$ that are joined by a node corresponding to $\alpha$ (which means that the associated partition of $\{1,\dots,n\}$ is $\Sc$). Without loss of generality, we can assume that $C^1_\yy$ is the component that degenerates to $C^{r+1}_\xx$.  By construction, the restriction of $\Phi_{r+1}$ to $C^1_\yy$ has zeros of order $k_i$ at the $i$-marked point, with $i\in I_1$, and a pole of order $d$ at the node. The residue of $\Phi_{r+1}$ at the node gives a well defined non-vanishing holomorphic function $f$ on $\widehat{\Uc}^r_\xx\cap D^*_\Sc$. This function extends to non-vanishing function on $\widehat{\Uc}^r_\xx\cap \hat{D}_\Sc$, which will be also denoted by $f$. We now observe that $f^{-1}\cdot\Phi_{r+1}$ coincides with the section $\Phi$ in $D^*_\Sc\cap\widehat{\Uc}^r_\xx$. Since $f^{-1}\cdot\Phi_{r+1}$ is a trivializing section of $\hat{\Lc}_\mu$ on $\hat{D}_\Sc\cap\widehat{\Uc}^r_\xx$, we conclude that the vanishing order of $\Phi$ along $\hat{D}_{\Sc'\cdot\Sc}=\hat{D}_\Sc\cap\hat{D}_{\Sc'}$ is zero.
\end{proof}

\begin{Lemma}\label{lm:div:P0:order:sigma:non:0}
If $i=0$, then the vanishing order of $\Phi$ along $\hat{D}_{\Sc'\cdot\Sc}$ is $-m(\Sc')$, where
$$
m(\Sc')=d^r\prod_{j=1}^r(\mu(I'_j)-1)=m_1\cdots m_r.
$$
\end{Lemma}
\begin{proof}
Let $\Phi_0,\dots,\Phi_{r+1}$ be the differentials defined in \cite[\textsection 2]{Ng23:a}.
Consider a point $\yy\in \Uc_\xx\cap D^*_\Sc$. By definition, $C_\yy$ has two components $C^0_\yy$ and $C^1_\yy$ separated by a node associated to the partition  $\Sc=\{I_0,I_1\}$. There is a degenerating map $\varphi_\yy: C_\yy \to C_\xx$ which sends  $C^0_\yy$ and $C^1_\yy$ onto the subcurves of $C_\xx$ that are separated by  the node $\alpha$. We can assume that $C^0_\xx \subset \varphi_\yy(C^0_\yy)$.

Since $C^0_\xx \subset \varphi_\yy(C^0_\yy)$, the restriction of  $\Phi_0$ to $C^0_\yy$ is a non-zero $d$-differential having a pole of order $d$ at the node. Multiplying $\Phi_0$ by a non-vanishing regular function on $\Uc_\xx\cap D_\Sc$, we assume that the residue of $\Phi_0$ at the node associated to $\Sc$ is $1$, which means that $\Phi_0$ is an extension of $\Phi$ to $\widehat{\Uc}_\xx\cap \hat{D}_\Sc$.
Consider $\Phi_0$ as a meromorphic section of $\hat{\Lc}_\mu$ over $\widehat{\Uc}_\xx\cap \hat{D}_\Sc$. We wish to calculate the vanishing order of this section along the divisor $\hat{D}_{\Sc'\cdot\Sc}$ of $\hat{D}_{\Sc}$.
It is enough to restrict our calculation to the open affine $\widehat{\Uc}^r_\xx$.
By the results of \cite{Ng23:a}, we know that $\Phi_r$ is a trivializing section of $\hat{\Lc}_\mu$ on $\hat{\Uc}^r_\xx$, and $\Phi_0/\Phi_r=t^{\beta_r}/t^{\beta_0}=t_r^{-m_r}$.
Thus, the vanishing order of $\Phi_0$ along $\hat{D}_{\Sc'\cdot\Sc}$ is equal to the vanishing order of $t_r^{-m_r}$ along the divisor $\hat{D}_{\Sc'}$. By Lemma~\cite[Lem. 6.1]{Ng23:a}, the order of $t_r$ along $\hat{D}_{\Sc'}$ is $\prod_{j=1}^{r-1}m_j$. Therefore the order of $t_r^{-m_r}$ along $\hat{D}_{\Sc'}$ is equal to $-\prod_{j=1}^rm_j=-m(\Sc')$, and the lemma follows.
\end{proof}
\subsubsection{Proof of Theorem~\ref{th:tauto:l:bdl:on:div:P0}}
\begin{proof}
By Proposition~\ref{prop:bdry:div:bl:P0}, we know that the boundary of $\hat{D}_\Sc$ consists of $\hat{D}_{\Sc'\cdot\Sc}$, with $\Sc'\in \hat{\Pc}(\mu,\Sc)$. Thus by combining Lemma~\ref{lm:div:P0:sigma:order:zero} and Lemma~\ref{lm:div:P0:order:sigma:non:0}, we get the desired conclusion.
\end{proof}


\section{Recursive formula}\label{sec:recursive}
\subsection{Statement of a recursive formula}\label{subsec:strate:recursive:form}
If  $d$ does not divide $k_i$, for all $i=1,\dots,n$, then it follows  from Theorem~\ref{th:vol:n:inters:nbr} that we have
\begin{equation}\label{eq:vol:n:inters:g0}
\vol_1(\pstratesp)=\frac{(-1)^{n-3}}{d^{n-3}}\cdot\frac{(2\pi)^{n-2}}{2^{n-2}(n-2)!}\cdot c_1(\hat{\Lc}_\mu)^{n-3}.
\end{equation}
By Theorem~\ref{th:tauto:div} (b), on the right hand side of \eqref{eq:vol:n:inters:g0} one can replace $c_1(\hat{\Lc}_\mu)^{n-3}$ by the self-intersection number of $\hat{\Dc}_\mu$ given in \eqref{eq:tauto:ln:bdl:expr}.

For large values  of $n$, a direct computation of $\hat{\Dc}_\mu^{n-3}$ by intersection theory seems to be an unrealistic task. Therefore, a recursive formula which relates $\hat{\Dc}_\mu^{n-3}$ to the self-intersection numbers of the tautological divisors on some strata of lower dimensions is more useful in this situation. To state the result, let us fix some notation. For each $k\in \Z_{\geq 3}$, let
$$
L_k=\{\nu=(\nu_1,\dots,\nu_k)\in \Q^k, \; \sum_{i=1}^k \nu_i = 2, \, \text{ and } \nu_i <1, \text{ for all } i=1,\dots,k\}.
$$
Note that we allow some of the entries of $\nu\in L_k$ to be zero. We define a map $\Jc_k: L_k \to \Z$ as follows: given $\nu=(\nu_1,\dots,\nu_k)\in L_k$, there is a smallest natural number $e$ such that $e\nu_i\in \Z$, for all $i=1,\dots,k$. Set $\ell_i=-e\nu_i$.
We then have $\ell_1+\dots+\ell_k=-2e, \; \gcd(e,\ell_1,\dots,\ell_k)=1$, and $\ell_i>-e, \; i=1,\dots,k$.
Let $\ul{\ell}:=(\ell_1,\dots,\ell_k)$, and $\Omega^e\Mod_{0,k}(\ul{\ell})$  the stratum of $e$-differentials associated to the vector $\ul{\ell}$.
By \cite{Ng23:a}, we know that the closure of $\Pb\Omega^e\Mod_{0,k}(\ul{\ell})$ in $\Pb\ol{\Hc}^{(e)}_{0,k}$ is isomorphic to the blow-up  $\widehat{\Mod}_{0,k}(\ul{\ell})$ of $\ol{\Mod}_{0,k}$.
We  then define
\begin{equation}\label{eq:def:funct:Jk}
\Jc_k(\nu) :=c_1^{k-3}(\hat{\Lcal}_\nu)\cap[\widehat{\Mod}_{0,k}(\ul{\ell})],
\end{equation}
where $\hat{\Lcal}_\nu$ is the tautological line bundle over $\widehat{\Mod}_{0,k}(\ul{\ell})$.

We are now back to the family of weights $\mu=(\mu_1,\dots,\mu_n)$. In what follows, for all subset $I$ of $\{1,\dots,n\}$, we define $\mu(I):=\sum_{i\in I}\mu_i$.  Consider the following families of  partitions of $\{1,\dots,n\}$ which depend on $\mu$.
\begin{itemize}
\item[$\bullet$] $\Tc_{1a}(\mu)$ is the set of partitions $\{I_0,I_1\}$, where $\mu(I_0) < 1 <  \mu(I_1)$, $|I_0|=2$ and $\mu(I_0) \not\in \Z$.

\item[$\bullet$] $\Tc_{1b}(\mu)$ is the set of partitions $\{I_{00}, I_{01}, I_1\}$ such that
	\begin{itemize}
	    \item[.]  $|I_{00}|=2, \; \mu(I_{00})=1$,
	
	    \item[.]  $|I_{01}|=1, \; \mu(I_{01})<0$,
	
		\item[.] $\mu(I_1)>1$, and $\mu(I_1)\not\in\Z$.
	\end{itemize}

\item $\Tc_{2a}(\mu)$ is the set of  partitions $\{I_0,I_1,I_2\}$, where $I_0,I_1,I_2$ satisfy
	\begin{itemize}
		\item[.] $|I_0|=1, \; \mu(I_0)< 0$,
		
		\item[.] $\mu(I_j) >1$, and  $\mu(I_j)\not\in\Z$, for $j=1,2$.
	\end{itemize}

\item[$\bullet$] $\Tc_{2b}(\mu)$ is the set of  partitions $\{I_{01},I_{02},I_1,I_2\}$ of $\{1,\dots,n\}$ such that
    \begin{itemize}
    \item[.] $|I_{01}|=|I_{02}|=1$,

    \item[.] $\mu(I_j)>1$, and $\mu(I_j)\not\in\Z$ for $j=1,2$,

    \item[.] $\mu(I_1)+\mu(I_{01})=\mu(I_2)+\mu(I_{02})=1$.
    \end{itemize}
\end{itemize}
We denote by $\Tc(\mu)$ the union  $\Tc_{1a}(\mu)\cup\Tc_{1b}(\mu)\cup\Tc_{2a}(\mu)\cup\Tc_{2b}(\mu)$. Elements of $\Tc(\mu)$ are called {\em primary  partitions}  with respect to $\mu$.
We associate to elements of $\Tc(\mu)$ the  following parameters:
\begin{itemize}
\item[$\bullet$] For $\Sc=\{I_0,I_1\} \in \Tc_{1a}(\mu)$ or $\Sc=\{I_{00},I_{01}, I_1\}\in \Tc_{1b}(\mu)$, we define
	\begin{itemize}
		\item[.] $e(\Sc)$ is the smallest natural number such that $e(\Sc)\mu_i \in \Z$, for all $i\in I_1$,
		
		\item[.] $m(\Sc):=d(\mu(I_1)-1)$,
		
		\item[.] $\nu(\Sc)$ is the vector in $\Q^{|I_1|+1}$ whose first coordinate is $2-\mu(I_1)$, and the remaining coordinates are given by $\{\mu_i, \; i \in I_1\}$.
	\end{itemize}

\item[$\bullet$] For $\Sc=\{I_0,I_1,I_2\} \in \Tc_{2a}(\mu)$ or $\Sc=\{I_{01},I_{02},I_1,I_2\} \in \Tc_{2b}(\mu)$, define
	\begin{itemize}
		\item[.] $n_1(\Sc)=|I_1|, n_2(\Sc)=|I_2|$,
		
		\item[.] $m_1(\Sc)=d(\mu(I_1)-1), m_2(\Sc)=d(\mu(I_2)-1)$, $m(\Sc)=m_1(\Sc)m_2(\Sc)$,
		
		\item[.] for $j\in \{1,2\}$, $e_j(\Sc)$ is the smallest natural number such that $e_j(\Sc)\mu_i \in \Z$ for all $i\in I_j$,
		
    	\item[.] $\nu_j(\Sc)$ is the vector in $\Q^{|I_j|+1}$ whose first coordinate 	is $2-\mu(I_j)$, and the remaining coordinates are given by $\{\mu_i, \; i \in I_j\}$.
		
		\end{itemize}

If $\Sc\in \Tc_{2b}(\mu)$, we define
		$$
		\eps(\Sc):=\left\{\begin{array}{cl}
		1 & \text{ if } \mu(I_{01}) \neq \mu(I_{02}) \\
		2 & \text{ if } \mu(I_{01})= \mu(I_{02}).
        \end{array}		
		\right.
		$$
\end{itemize}

\begin{Theorem}\label{th:recursive:1}
If there exists $i\in \{1,\dots,n\}$ such that $\mu_i\in \Z$, then $\Jc_n(\mu)=0$. Otherwise, we have
\begin{eqnarray}\label{eq:recursive:1}
\Jc_n(\mu) & = & \sum_{\Sc \in \Tc_{1a}(\mu)} a_\Sc\cdot\left(\frac{d}{e(\Sc)}\right)^{n-4}\cdot \Jc_{n-1}(\nu(\Sc)) - \sum_{\Sc \in \Tc_{1b}(\mu)} a_\Sc\cdot\left(\frac{d}{e(\Sc)}\right)^{n-5}\cdot \Jc_{n-2}(\nu(\Sc)) \\
\nonumber &  & -\sum_{\Sc \in \Tc_{2a}(\mu)} a_\Sc\cdot\frac{d^{n-5}}{e^{n_1(\Sc)-2}_1(\Sc)e^{n_2(\Sc)-2}_2(\Sc)}\cdot \Jc_{n_1(\Sc)+1}(\nu_1(\Sc))\Jc_{n_2(\Sc)+1}(\nu_2(\Sc))\\
\nonumber &  & +\sum_{\Sc \in \Tc_{2b}(\mu)} a_\Sc\cdot\frac{d^{n-6}}{e^{n_1(\Sc)-2}_1(\Sc)e^{n_2(\Sc)-2}_2(\Sc)}\cdot \Jc_{n_1(\Sc)+1}(\nu_1(\Sc))\Jc_{n_2(\Sc)+1}(\nu_2(\Sc)),
\end{eqnarray}
where the coefficients $a_\Sc$ are given by
$$
a_\Sc=\left\{
\begin{array}{ll}
\frac{d(n-3)}{(n-1)(n-2)}-\frac{m}{n-2} & \text{ if } \Sc\in \Tc_{1a}(\mu) \\
\frac{d(n-3)}{(n-1)(n-2)}\cdot m & \text{ if } \Sc\in \Tc_{1b}(\mu)\\
\frac{d(n_1n_2(m_1+m_2)-m_1n_1-m_2n_2)}{(n-1)(n-2)} -\frac{m}{n-2} & \text{ if } \Sc \in \Tc_{2a}(\mu) \\
\eps\cdot \frac{dn_1n_2}{(n-1)(n-2)}\cdot m & \text{ if } \Sc\in \Tc_{2b}(\mu).
\end{array}
\right.
$$
(for simplicity, we write $n_i,m_i,m, \eps$ in the place of $n_i(\Sc),m_i(\Sc),m(\Sc), \eps(\Sc)$ respectively).
\end{Theorem}
The proof of this theorem will occupy the rest of this section.

\subsection{Self-intersection of the tautological divisors on boundary components}\label{subsec:self:inters:on:bdry:div}
To prove \eqref{eq:recursive:1}, we first  prove the formulas computing the self-intersection numbers of divisors representing the restriction of the tautological line bundle on the components of $\partial\blowupsp$.
Consider a divisor $\hat{D}_\Scal$ in $\partial\blowupsp$ which is associated to a partition $\Scal=\{I_0,I_1,\dots,I_r\}\in \hat{\Pcal}_1(\mu)$. Let $n_j:=|I_j|, \; j=0,\dots,r$.
For $j=1,\dots,r$, let $\nu^{(j)}=(\nu^{(j)}_1,\dots,\nu^{(j)}_{n_j+1})$ be the vector whose first coordinate is $2-\mu(I_j)$, and the remaining coordinates are given by $\{\mu_i, \; i\in I_j\}$ (the ordering of $\{\mu_i, i \in I_j\}$ is unimportant).
Let $e_j$ be the smallest natural number such that $\ell^{(j)}_k:=-e_j\nu^{(j)}_k \in \Z$, and $\ul{\ell}^{(j)}=(\ell_1^{(j)},\dots,\ell^{(j)}_{n_j+1})$.
Recall that $N=n-3=\dim \blowupsp$, and we must have $n_0+r\geq 3$.

\begin{Proposition}\label{prop:int:Chern:cl:div:P1}
If $n_0+r>3$ then
$$
c_1^{N-1}(\OO(-1))\cap[\hat{D}_\Scal]=0.
$$
Assume that $n_0+r=3$, then we must have $r\in \{1,2\}$,  and
\begin{equation}\label{eq:int:Chern:cl:on:div:P1}
c_1^{N-1}(\OO(-1))\cap[\hat{D}_\Scal]=(-1)^{r+1}\frac{d^{N-r}}{e_1^{n_1-2}\cdot\dots \cdot e_r^{n_r-2}}\prod_{j=1}^r\Jc_{n_j+1}(\nu^{(j)}).
\end{equation}
In particular,  $ c_1^{N-1}(\OO(-1))\cap[\hat{D}_\Scal]\neq 0$ only if $\Sc\in \Tc_{1a}\cup\Tc_{2a}$.
\end{Proposition}

\begin{proof}
Define $\kappa^{(j)}:=-d\cdot\nu^{(j)}$ for all $j\in \{1,\dots,r\}$.
Let $\hat{\Lcal}_j$ be the tautological line bundle on $\widehat{\Mod}_{0,n_j+1}(\kappa^{(j)})$.
Recall that $E_\Sc$ is the vector bundle $\hat{\Lc}_{1}\oplus\dots\oplus\hat{\Lc}_{r}$ over
$$
\Mod_\Scal=\ol{\Mod}_{0,n_0+r}\times\left(\prod_{j=1}^r\widehat{\Mod}_{0,n_j+1}(\kappa^{(j)})\right)
$$
(here we abusively denote by $\hat{\Lcal}_{j}$ the pullback of the tautological line bundle on $\widehat{\Mod}_{0,n_j+1}(\kappa^{(j)})$ to $\Mod_\Scal$).
By Proposition~\ref{prop:bdry:div:bl:P1}, there is a surjective birational morphism
$i_\Scal: \Pb E_\Scal \to \hat{D}_\Scal$, which satisfies $i^*_\Scal\OO(-1) \sim \OO(-1)_{\Pb E_\Scal}$.
It follows that
$$
c^{N-1}_1(\OO(-1))\cap[\hat{D}_\Scal]=c_1^{N-1}(\OO(-1)_{\Pb E_\Scal})\cap[\Pb E_\Scal]=(-1)^{N-1}c_1^{N-1}(\OO(1)_{\Pb E_\Scal})\cap[\Pb E_\Scal].
$$
By definition, we have
$$
c_1^{N-1}(\OO(1)_{\Pb E_\Scal})\cap[\Pb E_\Scal]=s_{N-r}(E_\Scal)\cap[\Mod_\Scal]
$$
where $s_{N-r}$ is the $(N-r)$-th Segre class of $E_\Scal$ (see \cite[\textsection 3.1]{Ful}).
Since $E_\Scal$ is the direct sum of the line bundles $\hat{\Lcal}_1,\dots,\hat{\Lcal}_r$, the total Segre class of $E_\Scal$ is given by
$s(E_\Scal)=\prod_{j=1}^r(1+c_1(\hat{\Lcal}_j))^{-1}$.
It follows that
$$
s_{N-r}(E_\Scal)=(-1)^{N-r}\sum_{\alpha_1+\dots+\alpha_r=N-r}c_1^{\alpha_1}(\hat{\Lcal}_1)\cdots c_1^{\alpha_r}(\hat{\Lcal}_r).
$$
Thus we have
\begin{equation}\label{eq:Segre:cl:prod:Chern:cl}
s_{N-r}(E_\Scal)\cap[\Mod_\Scal]=(-1)^{N-r}\sum_{\alpha_1+\dots+\alpha_r= N-r}c_1^{\alpha_1}(\hat{\Lcal}_1)\cap[\widehat{\Mod}_{0,n_1+1}(\kappa^{(1)})]\cdot\dots\cdot c_1^{\alpha_r}(\hat{\Lcal}_r)\cap[\widehat{\Mod}_{0,n_r+1}(\kappa^{(r)})].
\end{equation}
We now remark that
$$
\dim\widehat{\Mod}_{0,n_1+1}(\kappa^{(1)})+\dots+\dim\widehat{\Mod}_{0,n_r+1}(\kappa^{(r)})= \sum_{j=1}^r n_j -2r=n-n_0-2r,
$$
while $N-r=n-3-r$. Therefore, if $n_0+r >3$ then $N-r >n-n_0-2r$, hence $s_{N-r}(E_\Scal)\cap[\Mod_\Scal]=0$.

Since we must have $n_0+r\geq 3$, it remains to consider the case $n_0+r=3$.
By assumption, $r\geq 1$, hence we have $r\in\{1,2,3\}$.
We first claim that $r\neq 3$. Indeed, if  $r=3$ then  $n_0=0$, which means that $I_0=\varnothing$. Since we have $\sum_{i\in I_j}\mu_i >1$, for all $j=1,2,3$, this implies that $\mu_1+\dots+\mu_n>3$, which is a contradiction.

Assume now that $r=2$. Then \eqref{eq:Segre:cl:prod:Chern:cl} gives
\begin{equation*}
s_{N-2}(E_\Scal)\cap[\Mod_\Scal]= (-1)^{N-2}c_1^{n_1-2}(\hat{\Lcal}_1)\cap[\widehat{\Mod}_{0,n_1+1}(\kappa^{(1)})]\cdot c_1^{n_2-2}(\hat{\Lcal}_2)\cap[\widehat{\Mod}_{0,n_2+1}(\kappa^{(2)})].
\end{equation*}
Recall that $e_j$ is the smallest natural number such that $\ell^{(j)}:=-e_j\nu^{(j)} \in \Z^{n_j+1}$.
Note that we have $\kappa^{(j)}=(d/e_j)\cdot \ell^{(j)}$.
The compactification of $\Pb\Omega^{e_j}\Mod_{0,n_j+1}(\ell^{(j)})$ in $\Pb\ol{\Hcal}^{(e_j)}_{0,n_j+1}$ is isomorphic to $\widehat{\Mod}_{0,n_j+1}(\ell^{(j)})$.
Let $\hat{\Lcal}_{\nu^{(j)}}$ be the tautological line bundle over $\widehat{\Mod}_{0,n_j+1}(\ell^{(j)})$.
We have a natural surjective birational morphism $f_j: \widehat{\Mod}_{0,n_j+1}(\ell^{(j)}) \to \widehat{\Mod}_{0,n_j+1}(\kappa^{(j)})$ which is induced by the map $(C,q) \mapsto (C,q^{\otimes(d/e_j)})$, for all $(C,q) \in \Omega^{e_j}\Mod_{0,n_j+1}(\ell^{(j)})$.
By construction, we have $f^*_j\hat{\Lcal}_j\sim \hat{\Lcal}^{\otimes\frac{d}{e_j}}_{\nu^{(j)}}$.
Thus
$$
c_1^{n_j-2}(\hat{\Lcal}_j)\cap[\widehat{\Mod}_{0,n_j+1}(\kappa^{(j)})]=\left(\frac{d}{e_j}\right)^{n_j-2}\cdot c^{n_j-2}_1(\hat{\Lcal}_{{\nu}^{(j)}})\cap[\widehat{\Mod}_{0,n_j+1}(\ell^{(j)})]= \left(\frac{d}{e_j}\right)^{n_j-2} \Jc_{n_j+1}(\nu^{(j)}).
$$
Hence \eqref{eq:int:Chern:cl:on:div:P1} is proved for this case. The case $n_0=2, r=1$ follows from the same arguments.
\end{proof}

We now turn to the case $\hat{D}_\Scal$ with $\Scal=\{I_0,I_1\}\in \hat{\Pcal}_0(\mu)$.
By Theorem~\ref{th:tauto:l:bdl:on:div:P0}, the restriction of $\OO(-1)$ to $\hat{D}_{\Scal}$ is represented by an explicit Weil divisor with support in the union of divisors $\hat{D}_{\Scal'\cdot\Scal}$, with $\Scal'\in \hat{\Pcal}^*(\mu,\Scal)$.
Given $\Sc'=\{I'_0,I'_1,\dots,I'_r\} \in \hat{\Pc}^*(\mu,\Sc)$, we define  $\Sc'*\Sc:= \{I'_0\cap I_0, I'_{0}\cap I_1,I'_1,\dots,I'_r\}$.
Remark that $\Sc'*\Sc$ is the partition associated to the stratum $D^*_{\Sc'\cdot\Sc}$ in $\ol{\Mod}_{0,n}$ (see Definition~\ref{def:part:div:in:div:P0}), however  $D^*_{\Sc'\cdot\Sc}$ is not uniquely determined by $\Sc'*\Sc$.


\begin{Proposition}\label{prop:int:Chern:cl:on:div:P0}
Let $\Scal'=\{I'_0,\dots,I'_r\} \in \hat{\Pcal}^*(\mu,\Scal)$. If $n_0+r >4$ then
$$
c_1^{N-2}(\OO(-1))\cap[\hat{D}_{\Scal'\cdot\Scal}]=0.
$$
If $n_0+r=4$, then we must have $r\in \{1,2\}$, and
\begin{equation}\label{eq:int:Chern:cl:on:div:P0}
c_1^{N-2}(\OO(-1))\cap[\hat{D}_{\Scal'\cdot\Scal}] = (-1)^{r+1}\frac{d^{N-r-1}}{e_1^{n_1-2}\cdots e_r^{n_r-2}}\prod_{j=1}^r\Jc_{n_j+1}(\mu^{(j)}).
\end{equation}
In particular, $c_1^{N-2}(\OO(-1))\cap[\hat{D}_{\Sc'\cdot\Sc}]=0$ unless $\Sc'*\Sc \in \Tc_{1b}\cup\Tc_{2b}$.
\end{Proposition}
\begin{proof}
Let  $I'_{00}=I'_0\cap I_0$ and $I'_{01}\cap I_1$.
The dual graph $\T_{\Scal'\cdot\Scal}$ associated to the stratum $D^*_{\Sc'\cdot\Sc}$ can be constructed as follows: the vertices of $\T_{\Scal'\cdot\Scal}$ are labeled by $\{I'_{00},I'_{01},I'_1,\dots,I'_r\}$. There is an edge between $I'_{00}$ and $I'_{01}$.
For all  $j\in \{1,\dots,r\}$, there is an edge between $I'_j$ and $I'_{00}$ if and only if $I'_j \subset I_0$, otherwise there is an edge between $I'_j$ and $I_{01}$.
We set $r_0$ (resp. $r_1$) to be the number of vertices among $\{I'_1,\dots,I'_r\}$ that are connected to $I'_{00}$ (resp. to $I'_{01}$) by an edge. Note that we have $r=r_0+r_1$.

Let $n_{00}=|I'_{00}|, n_{01}=|I'_{01}|$, and $n_j=|I'_j|, \; j=1,\dots,r$.
Define the vector $\kappa^{(j)}\in \Z^{n_j+1}$ by setting $\kappa_1^{(j)}=-d(2-\mu(I'_j))$, and the remaining coordinates given by $\{-d\mu_i, \; i\in I'_j\}$. Consider
$$
\Mod_{\Scal'\cdot\Scal}=\ol{\Mod}_{0,n_{00}+r_0+1}\times\ol{\Mod}_{0,n_{01}+r_1+1}\times \prod_{j=1}^r\widehat{\Mod}_{0,n_j+1}(\kappa^{(j)}).
$$
Note that $\Mod_{\Sc'\cdot\Sc}$ is a divisor in  $\Mod_{\Sc'}$.
Let $E_{\Sc'\cdot\Sc}$ denote the restriction of the vector bundle $E_{\Sc'}$ to $\Mod_{\Scal'\cdot\Scal}$. By the same arguments as in Proposition~\ref{prop:bdry:div:bl:P1}, we see that there is a surjective birational morphism $i_{\Scal'\cdot\Scal}: \Pb E_{\Scal'\cdot\Scal} \to \hat{D}_{\Scal'\cdot\Scal}$ such that $i^*_{\Scal'\cdot\Scal}\OO(-1) \sim \OO(-1)_{\Pb E_{\Scal'\cdot\Scal}}$. Thus, by the same arguments as in Proposition~\ref{prop:int:Chern:cl:div:P1}, we have
\begin{eqnarray*}
c_1^{N-2}(\OO(-1))\cap[\hat{D}_{\Scal'\cdot\Scal}] & = & c_1^{N-2}(\OO(-1)_{\Pb E_{\Scal'\cdot\Scal}})\cap [\Pb E_{\Scal'\cdot\Scal}]\\
& = & (-1)^{N-2}c_1^{N-2}(\OO(1)_{\Pb E_{\Scal'\cdot\Scal}})\cap [\Pb E_{\Scal'\cdot\Scal}]\\
& = & (-1)^{N-2}s_{N-r-1}(E_{\Scal'\cdot\Scal})\cap [\Mod_{\Scal'\cdot\Scal}]\\
& = & (-1)^{r+1}\sum_{\alpha_1+\dots+\alpha_r=N-r-1} \prod_{j=1}^r c_1^{\alpha_j}(\hat{\Lcal}_j)\cap[\widehat{\Mod}_{0,n_j+1}(\kappa^{(j)})].
\end{eqnarray*}
We now observe that
$$
\dim\widehat{\Mod}_{0,n_1+1}(\kappa^{(1)})+\dots+\dim\widehat{\Mod}_{0,n_r+1}(\kappa^{(r)}) = \sum_{j=1}^rn_j -2r = n-n_0-2r,
$$
while $N-r-1=n-r-4$.
Recall that we must have $n_{00}+r_0+1 \geq 3$ and $n_{01}+r_1+1 \geq 3$, which implies that $n_0+r=n_{00}+n_{01}+r_0+r_1\geq 4$.
If $n_0+r>4$ then $N-r-1 > n-n_0-2r$, hence $c_1^{N-2}(\OO(-1))\cap[\hat{D}_{\Scal'\cdot\Scal}]=0$.  Thus it suffices to consider the case $n_0+r=4$.
Since $r\geq 1$, we have $r\in\{1,2,3,4\}$.
We claim that $r\in\{1,2\}$. Indeed, if $r=4$ then $n_0=0$, which means that $I'_0=\varnothing$. We then have
$$
\sum_{i=1}^n\mu_i=\sum_{j=1}^4\sum_{i\in I_j}\mu_i >4
$$
which is impossible. If $r=3$ then $n_0=1$, and we can suppose that $n_{00}=0$ and $n_{01}=1$. Since we must have $n_{00}+r_0+1=3$, it follows that $r_0=2$. Without loss of generality, we can assume that $I_0=I'_1\cup I'_{2}$. But we would have
$$
1=\sum_{i\in I_0}\mu_i = \sum_{j=1}^{2}\sum_{i\in I'_j}\mu_i >2
$$
which is again impossible. Therefore, we must have $r\in \{1,2\}$.
If  $r=2$, then $n_0=2$, and the conditions $n_{00}+r_0+1=n_{01}+r_1+1=3$ implies that $r_0=r_1=1$.
If $r=1$, then $n_0=3$. We can assume that $r_0=1$ and $r_1=0$, which implies that $n_{00}=1$ and $n_{01}=2$.
The rest of the proof follows the same lines as Proposition~\ref{prop:int:Chern:cl:div:P1}.
\end{proof}

\subsection{Proof of Theorem~\ref{th:recursive:1}}\label{subsec:prf:recursive:form}
We first show
\begin{Lemma}\label{lm:ki:integer:I:zero}
Let $\nu:=(\nu_1,\dots,\nu_k)\in L_k$. If there exists $i\in \{1,\dots,k\}$ such that $\nu_i\in \Z$ then $\Jc_k(\nu)=0$.
\end{Lemma}
\begin{proof}
The tautological line bundle $\OO(-1)_{\Pb\Hc^{(e)}_{0,k}}$ has a natural Hermitian norm given by the area of the associated flat surfaces. We will call this norm the Hodge norm of $\OO(-1)_{\Pb\Hc^{(e)}_{0,k}}$, and denote by $\Theta$ its the curvature form.
The results of \cite{CMZ19} implies that for any $m \in \Z_{\geq 0}$, the restriction of $(\frac{\imath}{2\pi}\Theta)^{m}$ to a certain desingularization of $\widehat{\Mod}_{0,k}(\ul{\ell})$ is a representative in the sense of currents of $c^{m}_1(\hat{\Lcal}_\nu)$. In particular, we have
\begin{equation}\label{eq:inters:numb:n:int:curv:f}
c_1^{k-3}(\hat{\Lc}_\nu)\cap[\widehat{\Mod}_{0,n}(\ul{\ell})]=\int_{\Pb\Omega^e\Mod_{0,k}(\ul{\ell})} \left(\frac{\imath}{2\pi} \Theta\right)^{k-3}.
\end{equation}
If some of the entries of $\nu$ are integers, then the corresponding entries of the vector $\ul{\ell}$ are divisible by $e$.
In this case   $\Theta^{k-3}$ vanishes identically because some of the local coordinates are not involved the local expression of the Hodge norm. This means that the right hand side of \eqref{eq:inters:numb:n:int:curv:f} is zero, and the lemma follows.
\end{proof}

\begin{proof}[Proof of Theorem~\ref{th:recursive:1}]
The first assertion of Theorem~\ref{th:recursive:1} is obviously a consequence of Lemma~\ref{lm:ki:integer:I:zero}. To prove \eqref{eq:recursive:1} we will consider the restriction of the tautological line bundle on boundary divisors of $\blowupsp$.
By Theorem~\ref{th:tauto:div}, we know that $\hat{\Lc}_\mu \sim \hat{p}^*\Dc_\mu+\Ec$. Thus
$$
c_1^{N}(\hat{\Lc}_\mu)\cap[\blowupsp] = c_1^{N-1}(\hat{\Lc}_\mu)\cap[\hat{p}^*\Dc_\mu+\Ec].
$$
Using \eqref{eq:except:Weil:div}, we get an  explicit expression of $\hat{p}^*\Dc_\mu+\Ec$ as a combination of the boundary divisors $\{\hat{D}_{\Sc}, \; \Sc \in \hat{\Pc}(\mu)\}$. If we write $\hat{p}^*\Dc_\mu+\Ec = \sum_{\Sc\in \hat{\Pc}(\mu)} \lambda_\Sc\cdot \hat{D}_\Sc$, then
$$
c_1^{N-1}(\hat{\Lc}_\mu)\cap[\hat{p}^*\Dc_\mu+\Ec]=\sum_{\Sc\in \Pc(\mu)}\lambda_\Sc\cdot c_1^{N-1}(\hat{\Lc}_\mu)\cap[\hat{D}_\Sc].
$$
For $\Sc \in \hat{\Pc}_1(\mu)$, Proposition~\ref{prop:int:Chern:cl:div:P1} implies that $c_1^{N-1}(\hat{\Lc}_\mu)\cap \hat{D}_\Sc=0$ unless $\Sc$ belongs to $\Tc_{1a}$ or $\Tc_{2a}$.
We will calculate $\lambda_\Sc$ for all $\Sc\in \Tc_{1a}\cup\Tc_{2a}$.
\begin{itemize}
\item[$\bullet$] Case $\Sc=\{I_0,I_1\} \in \Tc_{1a}$. Recall that $|I_0|=2$, and $m(\Sc)=d\mu_\Sc$. Since $\hat{D}_\Sc$ is not contained in the support of $\Ec$, the coefficient of $\hat{D}_\Sc$ is given by  the coefficient of $D_\Sc$ in $\Dc_\mu$, that is
\[
\lambda_\Sc  =  \frac{d}{(n-2)(n-1)}(n-3-(n-1)\mu_\Sc) =  \frac{d(n-3)}{(n-2)(n-1)}-\frac{d\mu_\Sc}{n-2}  =  a_\Sc.
\]
\item[$\bullet$] Case $\Sc=\{I_0,I_1,I_2\} \in \Tc_{2a}$. Define $\Sc_1=\{I_0\cup I_2, I_1\}$ and $\Sc_2=\{I_0\cup I_1, I_2\}$. Then the closure $D_\Sc$ of the stratum $D^*_\Sc$ associated to $\Sc$  is the intersection of $D_{\Sc_1}$ and $D_{\Sc_2}$ in $\ol{\Mod}_{0,n}$. Therefore, $\hat{D}_\Sc$ is contained in the inverse images of both $D_{\Sc_1}$ and $D_{\Sc_2}$.
The coefficient of $D_{\Sc_1}$ in $\Dc_\mu$ is
\[
b_{\Sc_1}  = \frac{d}{(n-2)(n-1)}\cdot n_2 \cdot(n_1-1-(n-1)\mu_{\Sc_1}) =  \frac{dn_2(n_1-1)}{(n-1)(n-2)}-\frac{n_2m_1}{n-2}.
\]
Similarly, the coefficient of $D_{\Sc_2}$ in $\Dc_\mu$ is
$$
b_{\Sc_2}=\frac{dn_1(n_2-1)}{(n-1)(n-2)}-\frac{n_1m_2}{n-2}.
$$
Let $\xx$ be a point  in $D^*_\Sc$, and $t_1$ and $t_2$  the coordinate functions that define respectively $D_{\Sc_1}$ and $D_{\Sc_2}$ in the neighborhood  $\Uc_\xx$ of $\xx$.
By \cite[Lem. 6.1]{Ng23:a} we know that the order of $t_1$ along $\hat{D}_\Sc$ is equal to $m_2$. Thus the multiplicity of $\hat{D}_\Sc$ in $\hat{p}^*D_{\Sc_1}$ is $m_2$. Similarly, the multiplicity of $\hat{D}_\Sc$ in $\hat{p}^*D_{\Sc_2}$ is $m_1$.
It follows that the coefficient of $\hat{D}_\Sc$ in $\hat{p}^*\Dc_\mu$ is equal to
$$
m_2b_{\Sc_1}+m_1b_{\Sc_2}=\frac{d(m_2n_2(n_1-1)+m_1n_1(n_2-1))}{(n-1)(n-2)}-\frac{m_1m_2(n-1)}{n-2}
$$
(here we used $n_1+n_2=n-1$).
By \eqref{eq:except:Weil:div}, the coefficient of $\hat{D}_\Sc$ in $\Ec$  is $m_1m_2$. Therefore,
\[
\lambda_\Sc  =  \frac{d(m_2n_2(n_1-1)+m_1n_1(n_2-1))}{(n-1)(n-2)}-\frac{m_1m_2(n-1)}{n-2}+m_1m_2  =  a_\Sc.
\]
\end{itemize}
In conclusion, we have $\lambda_\Sc=a_\Sc$ for all $\Sc\in \Tc_{1a}\cup\Tc_{2a}$.

\medskip

Let us now consider a divisor $\hat{D}_{\Sc_0}$ associated to a partition $\Sc_0=\{I_0,I_1\} \in \hat{\Pc}_0(\mu)$. Recall that we have $\mu(I_0)=\mu(I_1)=1$. This implies that the coefficient of $D_{\Sc_0}$ in $\Dc_\mu$ is $\frac{d(|I_0|-1)(|I_1|-1)}{(n-1)(n-2)}$.
Since $\hat{D}_{\Sc_0}$ is not contained in the support of $\Ec$, we have
$$
\lambda_{\Sc_0}=\frac{d(|I_0|-1)(|I_1|-1)}{(n-1)(n-2)}.
$$
From Theorem~\ref{th:tauto:l:bdl:on:div:P0}, we have
$$
c_1(\hat{\Lc}_\mu)\cap[\hat{D}_{\Sc_0}] \sim \sum_{\Sc_1\in \hat{\Pc}^*(\mu,\Sc_0)} -m(\Sc_1)[\hat{D}_{\Sc_1\cdot\Sc_0}].
$$
Thus
\begin{equation}\label{eq:int:tauto:div:over:P0}
c_1^{N-1}(\hat{\Lc}_\mu)\cap\left(\sum_{\Sc_0\in \hat{\Pc}_0(\mu)}\lambda_{\Sc_0}\cdot[\hat{D}_{\Sc_0}]\right) = c_1^{N-2}(\hat{\Lc}_\mu)\cap\left(\sum_{\Sc_0\in \hat{\Pc}_0(\mu)}\sum_{\Sc_1\in\hat{\Pc}^*(\mu,\Sc_0)}-\lambda_{\Sc_0}\cdot m(\Sc_1)\cdot[\hat{D}_{\Sc_1\cdot\Sc_0}]\right).
\end{equation}
By Proposition~\ref{prop:int:Chern:cl:on:div:P0}, we know that $c_1^{N-2}(\OO(-1))\cap [\hat{D}_{\Sc_1\cdot\Sc_0}]=0$ unless the partition  $\Sc:=\Sc'*\Sc$ associated to $\hat{D}_{\Sc_1\cdot \Sc_0}$ belongs to either $\Tc_{1b}$ or $\Tc_{2b}$.

\begin{itemize}
\item[$\bullet$] If $\Sc=\{I_{00},I_{01},I_1\} \in \Tc_{1b}$, then there is a unique choice for $(\Sc_0,\Sc_1)$, namely $\Sc_0=\{I_{00}, I_{01}\cup I_1\}$, and $\Sc_1=\{I_{00}\cup I_{01}, I_1\}$.
Recall that we have $|I_{00}|=2$ and $|I_{01}|=1$. Therefore
$$
\lambda_{\Sc_0}m(\Sc_1)=\frac{d(n-3)}{(n-1)(n-2)}\cdot m(\Sc)=a_\Sc.
$$

\item[$\bullet$] For $\Sc=\{I_{01},I_{02},I_1,I_2\} \in \Tc_{2b}$, we have two cases: if $\mu(I_{01})\neq \mu(I_{02})$, then $\Sc_0$ and $\Sc_1$ are uniquely determined by $\Sc$, namely $\Sc_0=\{I_{01}\cup I_1, I_{02}\cup I_2\}$ and $\Sc_1=\{I_{01}\cup I_{02},I_1,I_2\}$.  If $\mu(I_{01})=\mu(I_{02})$, then the partition $\Sc$ arises from two pairs $(\Sc_0,\Sc_1)$ and $(\Sc'_0,\Sc_1)$, where $\Sc_0$ and $\Sc_1$ are as above, and $\Sc'_0=\{I_{01}\cup I_2, I_{02}\cup I_1\}$.
In either case, we have $\lambda_{\Sc_0} m(\Sc_1)=\lambda_{\Sc'_0} m(\Sc_1)=\frac{dn_1n_2}{(n-1)(n-2)}\cdot m(\Sc)$  and $ c_1^{N-2}(\hat{\Lc}_\mu)\cap [\hat{D}_{\Sc_1\cdot\Sc_0}]=c_1^{N-2}(\hat{\Lc}_\mu)\cap [\hat{D}_{\Sc'_1\cdot\Sc_0}]$.
\end{itemize}
Therefore
\begin{align*}
\sum_{\Sc_0\in \hat{\Pc}_0(\mu)} \lambda_{\Sc_0}\cdot c_1^{N-1}(\hat{\Lc}_\mu)\cap[\hat{D}_{\Sc_0}] & = -\sum_{\Sc\in\Tc_{1b}\cup\Tc_{2b}}a_\Sc\cdot c_1^{N-2}(\hat{\Lc}_\mu)\cap[\hat{D}_{\Sc_1\cdot\Sc_0}],
\end{align*}
where $\Sc_0$ and $\Sc_1$ are constructed from by $\Sc$ as above.
We can now apply Proposition~\ref{prop:int:Chern:cl:div:P1} and Proposition~\ref{prop:int:Chern:cl:on:div:P0} to conclude.
\end{proof}


\section{Volumes of strata of quadratic differentials}\label{sec:vol:strat:QD}
\subsection{Self-intersection numbers of the tautological divisors on strata of quadratic differentials}\label{subsec:self:inters:nb:QD:g0}
For all $\kappa=(k_1,\dots,k_n) \in (\Z_{\geq -1})^n$ such that $k_1+\dots+k_n=-4$, let us define
\begin{equation}\label{eq:def:vol:quad:str}
V(k_1,\dots,k_n):=\hat{\Dc}^{n-3}_\mu, \qquad \text{ where } \qquad \mu=(-\frac{k_1}{2},\dots,-\frac{k_n}{2}).
\end{equation}
As an application of Theorem~\ref{th:recursive:1}, we will prove
\begin{Theorem}\label{th:formula:d:2}
For all $(k_1,\dots,k_n)\in (\Z_{\geq -1})^n$ such that $k_i$ is odd for all $i=1,\dots,n$, and $k_1+\dots+k_n=-4$, we have
\begin{equation}\label{eq:vol:quad:form}
V(k_1,\dots,k_n)=(-1)^{\frac{n}{2}}\cdot(n-3)!\cdot\prod_{i=1}^n\frac{k_i!!}{(k_i+1)!!}
\end{equation}
where $k!!=k\cdot(k-2)\cdots3\cdot1$ for $k$  odd, and $k!!=k\cdot(k-2)\cdots2$ for $k$ even, with the convention $0!!=(-1)!!=1$.
\end{Theorem}

Theorem~\ref{th:vol:quad:diff:stra} is an immediate consequence of Theorem~\ref{th:formula:d:2}.
\begin{proof}[Proof of Theorem~\ref{th:vol:quad:diff:stra}]
By Theorem~\ref{th:vol:n:inters:nbr}, we have
$$
\vol_1(\Pb\Omega^2\Mod_{0,n}(\kappa)) = \frac{(-1)^{n-3}}{2^{n-3}}\cdot\frac{\pi^{n-2}}{(n-2)!}\cdot V(\kappa) = (-1)^{\frac{n-2}{2}}\cdot\frac{\pi^{n-2}}{2^{n-3}(n-3)}\cdot \prod_{i=1}^n\frac{k_i!!}{(k_i+1)!!} \quad \hbox{ by \eqref{eq:vol:quad:form}}.
$$
\end{proof}

In view of Theorem~\ref{th:recursive:1}, let us define $\Tc_{1b}(\kappa)$ to be the set of partitions $\Sc=\{I_{00},I_{01},I_1\}$ of $\{1,\dots,n\}$ such that
\begin{itemize}
\item[.] $I_{00}=\{i_0,i'_0\}$ such that $k_{i_0}=k_{i'_0}=-1$,

\item[.] $I_{01}=\{i_1\}$ such that $k_{i_1}>0$.
\end{itemize}
We define $m(\Sc):=k_{i_1}$, and $\kappa(\Sc)$ to be the vector in $\Z^{n-2}$ constructed as follows: the  first coordinate of $\kappa(\Sc)$ is $k_{i_1}-2$,  and the remaining coordinates of $\kappa(\Sc)$ are given by the entries of $\{k_i, i\in I_1\}$ (the order of $\{k_i,\; i \in I_1\}$ is unimportant). Observe that  all the entries of $\kappa(\Sc)$ are odd numbers.

Next, we define $\Tc_{2b}(\kappa)$ to be the  set of partitions $\Sc=\{I_{01},I_{02},I_{1},I_2\}$ such that
\begin{itemize}
\item[.] $I_{01}=\{i_1\}, I_{02}=\{i_2\}$ such that $k_{i_1}>0$ and $k_{i_2}>0$,

\item[.] $\sum_{i\in I_{1}}k_i+k_{i_1}=\sum_{i\in I_{2}}k_i+k_{i_2}=-2$.
\end{itemize}
We associate to each partition $\Sc\in \Tc_{2b}(\kappa)$ as above the following parameters
\begin{itemize}
\item[.] $m_1(\Sc):=k_{i_1}, m_2(\Sc):=k_{i_2}$,

\item[.] $n_1(\Sc)=|I_1|, \; n_2(\Sc)=|I_2|$,

\item[.] for $s\in\{1,2\}$, $\kappa_s(\Sc)$ is the vector in $\Z^{n_s+1}$ whose first coordinate is $k_{i_s}-2$, and the remaining $n_s$ coordinates are the entries of $I_s$.

\item[.] $\eps(\Sc):=\left\{
\begin{array}{ll}
1 & \hbox{ if $k_{i_1}\neq k_{i_2}$},\\
2 & \hbox{ if $k_{i_1}=k_{i_2}$}.
\end{array}
\right.
$
\end{itemize}
Let $\mu_i:=-\frac{k_i}{2}, \; i=1,\dots,n$, and $\mu:=(\mu_1,\dots,\mu_n)$. Define
\begin{equation}\label{eq:def:vol:quad:diff}
V(\kappa):=\Jc_n(\mu)=c_1^{n-3}(\OO(-1)_{\Pb\ol{\Hc}^{(2)}_{0,n}})\cap[\Pb\Omega^2\ol{\Mod}_{0,n}(\kappa)].
\end{equation}
Since $n$ must be even, one can easily see that $\Tc_{1a}(\mu)=\Tc_{2a}(\mu)=\vide$.
Moreover, by definition, we have that $\Tc_{1b}(\mu)=\Tc_{1b}(\kappa)$ and $\Tc_{2b}(\mu)=\Tc_{2b}(\kappa)$. Therefore, \eqref{eq:recursive:1} implies

\begin{Corollary}\label{cor:rec:form:quad:diff}
The numbers $V(\kappa)$ satisfy
\begin{eqnarray}
\label{eq:rec:formula:quad:form} V(\kappa) & = & -\sum_{\Sc \in \Tc_{1b}(\kappa)}\frac{2(n-3)}{(n-1)(n-2)}\cdot m(\Sc)\cdot V(\kappa(\Sc))+ \\
\nonumber & & +\sum_{\Sc\in \Tc_{2b}(\kappa)}\eps(\Sc)\frac{2n_1(\Sc)n_2(\Sc)}{(n-1)(n-2)}\cdot m_1(\Sc)\cdot m_2(\Sc)\cdot V(\kappa_1(\Sc))\cdot V(\kappa_2(\Sc)).
\end{eqnarray}
\end{Corollary}
We now set
$$
V^*(\kappa):=(-1)^{\frac{n}{2}}\cdot(n-3)!\prod_{i=1}^n\frac{k_i!!}{(k_i+1)!!}
$$
We have $V^*(-1,-1,-1,-1)=1=V(-1,-1,-1,-1)$ (see Appendix \ref{sec:low:dim:calcul}). Thus to prove Theorem~\ref{th:formula:d:2}, it is enough to show that the numbers $V^*(\kappa)$ satisfy the same recursive relation as $V(\kappa)$.
To this purpose, let us set
$$
\Cc_1(\kappa):=-\sum_{\Sc \in \Tc_{1b}(\kappa)}\frac{2(n-3)}{(n-1)(n-2)}\cdot m(\Sc)\cdot V^*(\kappa(\Sc))
$$
and
$$
\Cc_2(\kappa):=\sum_{\Sc\in \Tc_{2b}(\kappa)}\eps(\Sc)\frac{2n_1(\Sc)n_2(\Sc)}{(n-1)(n-2)}\cdot m_1(\Sc)\cdot m_2(\Sc)\cdot V^*(\kappa_1(\Sc))\cdot V^*(\kappa_2(\Sc)).
$$
Denote by $\Pc_0(\kappa)$ the set of $I\subset \{1,\dots,n\}$ such that $\sum_{i\in I} k_i=-2$.
We first show
\begin{Lemma}
\label{lm:contrib:T1b}
We have
\begin{equation}\label{eq:contrib:T1b}
\Cc_1(\kappa)=\sum_{\{i,i'\}\in \Pc_0(\kappa)} \frac{2(n-3)!}{(n-1)!}V^*(\kappa)= \frac{V^*(\kappa)}{(n-1)!}\sum_{\substack{I\in \Pc_0(\kappa), \\ |I|=2 \text{ or } |I^c|=2}}(|I|-1)!(|I^c|-1)!
\end{equation}
\end{Lemma}
\begin{proof}
Consider $\Sc=\{I_{00},I_{01},I_1\} \in \Tc_{1b}(\kappa)$. Let us assume that $I_{00}=\{1,2\}$, $I_{01}=\{3\}$. Recall that by definition, we have  $k_1=k_2=-1$, $m(\Sc)=k_3\geq 1$, and $\kappa(\Sc)=(k_3-2,k_4,\dots,k_n)$.
Therefore,
\begin{align*}
 m(\Sc)\cdot V^*(\kappa(\Sc))  &= k_3\cdot(-1)^{\frac{n-2}{2}}\cdot(n-5)!\cdot\frac{(k_3-2)!!}{(k_3-1)!!}\prod_{i=4}^n\frac{k_i!!}{(k_i+1)!!} \\
   &=(-1)^{\frac{n-2}{2}}\cdot (k_3+1)\cdot(n-5)!\prod_{i=3}^n\frac{k_i!!}{(k_i+1)!!}\\
   &=(-1)^{\frac{n-2}{2}}\cdot (k_3+1)\cdot(n-5)!\prod_{i=1}^n\frac{k_i!!}{(k_i+1)!!}\\
   &=-(k_3+1)\cdot \frac{(n-5)!}{(n-3)!}V^*(\kappa),
\end{align*}
where in the last two equalities, we make use of the convention
$$
\frac{k_1!!}{(k_1+1)!!}=\frac{k_2!!}{(k_2+1)!!}=\frac{(-1)!!}{0!!}=1.
$$
Thus, the contribution of $\Sc$ in $\Cc_1(\kappa)$ is equal to
\begin{align*}
\frac{2(n-3)}{(n-1)(n-2)}\cdot (k_3+1)\cdot\frac{(n-5)!}{(n-3)!}\cdot V^*(\kappa)& = (k_3+1)\cdot\frac{2(n-3)(n-5)!}{(n-1)!}\cdot V^*(\kappa).
\end{align*}
It follows that the contribution of the partitions $\Sc'=\{I'_{00},I'_{01},I'_1\}$ such that $I'_{00}=\{1,2\}$ in $\Cc_1(\kappa)$ is equal to
$$
\Cc_1(\kappa,\{1,2\}):=\left(\sum_{\substack{3\leq i \leq n\\ k_i >0}}(k_i+1)\right)\cdot \frac{2(n-3)(n-5)!}{(n-1)!}\cdot V^*(\kappa).
$$
Since for all $i\in \{1,\dots,n\}$, either $k_i=-1$ or $k_i>0$, we have
\begin{align*}
\sum_{\substack{3\leq i \leq n\\ k_i >0}}(k_i+1) & = \sum_{3\leq i \leq n}(k_i+1) = n-2+\sum_{3\leq i \leq n}k_i = n-4.
\end{align*}
Therefore
$$
\Cc_1(\kappa,\{1,2\})=\frac{2(n-3)!}{(n-1)!}\cdot V^*(\kappa).
$$
It follows that
$$
\Cc_1(\kappa)=\sum_{\substack{\{i,i'\}\subset \{1,\dots,n\}\\ k_i+k_{i'}=-2}}\Cc_1(\kappa,\{i,i'\}) = \sum_{\{i,i'\}\in \Pc_0(\kappa)}\frac{2(n-3)!}{(n-1)!}V^*(\kappa).
$$
\end{proof}

\begin{Lemma}\label{lm:contrib:T2b}
We have
\begin{equation}\label{eq:contrib:T2b}
\Cc_2(\kappa)=\frac{V^*(\kappa)}{(n-1)!}\cdot\sum_{\substack{I\in \Pc_0(\kappa),\\ 3 \leq |I| \leq n-3}}(|I|-1)!(|I^c|-1)!
\end{equation}
\end{Lemma}
\begin{proof}
Consider a partition $\Sc=\{I_{01},I_{02},I_1,I_2\} \in \Tc_{2b}(\kappa)$. Let $\hat{I}_s=I_{0s}\cup I_s$, for $s=1,2$. Recall that by definition, we have
$$
\sum_{i\in \hat{I}_1}k_i=\sum_{i\in \hat{I}_2}k_i=-2.
$$
To fix ideas, let us assume that $I_{01}=\{1\}$ and $I_{02}=\{2\}$. For simplicity, we will write $n_s,m_s,\kappa_s$ instead of $n_s(\Sc),m_s(\Sc), \kappa_s(\Sc), \; s=1,2$, respectively.
We have
\begin{align*}
m_1\cdot m_2\cdot V^*(\kappa_1)\cdot V^*(\kappa_2) &=  (-1)^{\frac{n_1+n_2+2}{2}}\cdot k_1\cdot k_2\cdot(n_1-2)!\cdot(n_2-2)!\cdot \frac{(k_1-2)!!}{(k_1-1)!!}\cdot \frac{(k_2-2)!!}{(k_2-1)!!}\cdot \prod_{i\geq 3}\frac{k_i!!}{(k_i+1)!!}\\
&= (-1)^{\frac{n}{2}}\cdot(k_1+1)\cdot(k_2+1)\cdot (n_1-2)!\cdot(n_2-2)!\cdot\prod_{i\geq 1}\frac{k_i!!}{(k_i+1)!!}\\
&= (k_1+1)\cdot(k_2+1)\cdot\frac{(n_1-2)!(n_2-2)!}{(n-3)!}V^*(\kappa).
\end{align*}
Let $\Cc_2(\kappa,(\hat{I}_1,\hat{I_2}))$ be the contribution in $\Cc_2(\kappa)$ of all the partitions $\Sc'=\{I'_{01},I'_{02},I'_1,I'_2\} \in \Tc_{2b}(\kappa)$ such that $I'_{01}\cup I'_{02}=\hat{I}_1$.

In the case $k_1=k_2$ we have $\eps(\Sc)=2$.  Since we do not have a preferred ordering of $I_{01},I_{02}$, we set the contribution of $\Sc$  to both $\Cc_2(\kappa,(I_{01}\cup I_1, I_{02}\cup I_2))$ and $\Cc_2(\kappa, (I_{01}\cup I_2, I_{02}\cup I_1))$ to be $\frac{2n_1n_2}{(n-1)(n-2)} m_1m_2V^*(\kappa_1)V^*(\kappa_2)$.
We then have
\begin{align*}
\Cc_2(\kappa,(\hat{I}_1,\hat{I}_2)) & = \frac{2n_1n_2}{(n-1)(n-2)}\cdot\frac{(n_1-2)!(n_2-2)!}{(n-3)!}\cdot V^*(\kappa)\cdot \sum_{\substack{k_{i_1} \in \hat{I}_1\\ k_{i_1}>0}}\sum_{\substack{k_{i_2}\in \hat{I}_2,  \\ k_{i_2}>0}}(k_{i_1}+1)(k_{i_2}+1)\\
&=\frac{2n_1n_2}{(n-1)(n-2)}\cdot\frac{(n_1-2)!(n_2-2)!}{(n-3)!}\cdot V^*(\kappa)\cdot \sum_{\substack{k_{i_1} \in \hat{I}_1}}\sum_{\substack{k_{i_2}\in \hat{I}_2}}(k_{i_1}+1)(k_{i_2}+1)\\
&=\frac{2n_1n_2(n_1-2)!(n_2-2)!V^*(\kappa)}{(n-1)!}\cdot\left(\sum_{\substack{k_{i_1} \in \hat{I}_1}}(k_{i_1}+1)\right)\left(\sum_{\substack{k_{i_2}\in \hat{I}_2}}(k_{i_1}+1) \right)\\
&=\frac{2n_1n_2(n_1-2)!(n_2-2)!V^*(\kappa)}{(n-1)!}\cdot(n_1-1)\cdot(n_2-1)\\
&=\frac{2n_1!n_2!}{(n-1)!}\cdot V^*(\kappa).
\end{align*}
(here we used the fact that either $k_i>0$ of $k_i=-1$ for all $i\in \{1,\dots,n\}$, and that $\sum_{i\in \hat{I}_1}k_i=\sum_{i\in \hat{I}_2}k_i=-2$).
Notice that $n_1=|\hat{I}_1|-1$ and $\hat{I}_2=\hat{I}^c_1$, we obtain
$$
\Cc_2(\kappa) =\sum_{\substack{\{I,I^c\}, \; I\in \Pc_0(\kappa)\\ |I| \geq 3, |I^c|\geq 3}}\Cc_2(\kappa,\{I,I^c\}) =\frac{V^*(\kappa)}{(n-1)!}\cdot\sum_{\substack{I\in \Pc_0(\kappa),\\ 3 \leq |I| \leq n-3}}(|I|-1)!(|I^c|-1)!
$$
\end{proof}

\begin{Proposition}\label{prop:contrib:total:b}
We have
\begin{equation}\label{eq:contrib:total:b}
\Cc_1(\kappa)+\Cc_2(\kappa)=\frac{V^*(\kappa)}{(n-1)!}\cdot \sum_{I\in \Pc_0(\kappa)}(|I|-1)!(|I^c|-1)!
\end{equation}
\end{Proposition}
\begin{proof}
Remark that for all $I \subset \{1,\dots,n\}$ such that $\sum_{i\in I}k_i=-2$, we must have $|I|\geq 2$.
Therefore this proposition is an immediate consequence of Lemma~\ref{lm:contrib:T1b} and Lemma~\ref{lm:contrib:T2b}.
\end{proof}

Let $P$ be a subset of $\{1,\dots,n\}$ that contains all the indices $i$ such that $k_i>0$ and at most two indices $i$ such that $k_i=-1$.
Note that $k_i=-1$ for all $i\in \{1,\dots,n\}\setminus P$.
For any subset $J\subset P$, let $s(J):=\sum_{i \in J} k_i$ and denote by $J^c$ the complement of $J$ in $P$.
Remark that we always have $s(J)\geq -2$.
\begin{Lemma}\label{lm:identity:(n-1)!:a}
Let $q:=n-|P|$. Then we have
\begin{align}
\label{eq:identity:(n-1)!:a} \sum_{I\in \Pc_0(\kappa)}(|I|-1)!(|I^c|-1)! & = q!\sum_{\substack{J\subset P}}\frac{(s(J)+|J|+1)!}{(s(J)+2)!}\cdot \frac{(s(J^c)+|J^c|+1)!}{(s(J^c)+2)!}\\
\nonumber & = \frac{q!}{2!(q-2)!}\cdot 2 \cdot (n-3)!+ q!\cdot F_{|P|,2,2}((k_i)_{i\in P}),
\end{align}
where $F_{|P|,2,2}$ is the symmetric function on $|P|$ variables defined in Appendix~\ref{sec:sym:poly:append}.
\end{Lemma}
\begin{proof}
Given  $J\subset P$, since $s(J)\geq -2$, one can always add to $J$ some subset of $P^c$, which can be empty, to get a subset in $\Pc_0(\kappa)$.
There are exactly $\left(\begin{array}{c} q \\ s(J)+2 \end{array}\right)$ subsets $I\in \Pc_0(\kappa)$ such that $J=I\cap P$. Note that for such $I$, we must have $|I|=|J|+s(J)+2$.
Therefore
\begin{align*}
\sum_{I\in \Pc_0(\kappa)}(|I|-1)!(|I^c|-1)! & = \sum_{J\subset P}\frac{q!}{(s(J)+2)!(q-s(J)-2)!}\cdot (|J|+s(J)+1)!(|J^c|+s(J^c)+1)!\\
& = \sum_{J\subset P}\frac{q!}{(s(J)+2)!(s(J^c)+2)!}\cdot (|J|+s(J)+1)!(|J^c|+s(J^c)+1)!\\
& = q!\sum_{J\subset P}\frac{(s(J)+|J|+1)!}{(s(J)+2)!}\cdot \frac{(s(J^c)+|J^c|+1)!}{(s(J^c)+2)!}
\end{align*}
here we used the fact that $s(J)+s(J^c)=q-4$.

Since $\sum_{i=1}^nk_i=s(P)-q=-4$, we have $s(P)=q-4$. Therefore, if $I=\varnothing$ or $I=P$, then the term associated to $I$ in \eqref{eq:identity:(n-1)!:a} is
$$
\frac{q!}{2!(q-2)!}\cdot 1!\cdot (s(P)+p+1)!=\frac{q!}{2!(q-2)!}\cdot (n-3)!
$$
By definition, we have
$$
F_{|P|,2,2}((k_i)_{i\in P})=\sum_{\substack{J\subset P, \\ J\not\in\{\varnothing, P\}}}\frac{(s(J)+|J|+1)!}{(s(J)+2)!}\cdot \frac{(s(J^c)+|J^c|+1)!}{(s(J^c)+2)!}.
$$
The lemma is then proved.
\end{proof}

\subsection{Proof of Theorem~\ref{th:formula:d:2}}
\begin{proof}
Theorem~\ref{th:formula:d:2}  follows from the following identity
\begin{equation}\label{eq:identity:(n-1)!}
\sum_{I\in \Pc_0(\kappa)}(|I|-1)!(|I^c|-1)!=(n-1)!
\end{equation}
Indeed, combined with Proposition~\ref{prop:contrib:total:b}, \eqref{eq:identity:(n-1)!} implies that the numbers $V(\kappa)$ and $V^*(\kappa)$ satisfy the same recursive relations. Since $V(-1^4)=V^*(-1^4)=1$, we must have $V(\kappa)=V^*(\kappa)$ for all $\kappa$.

Let $p$ be the number of positive entries in $\kappa$, and $q:=n-p$.
We will prove \eqref{eq:identity:(n-1)!} by induction on $p$.
If $p=1$, we can assume that $k_1$ is the unique positive entry of $\kappa$.
By taking $P=\{1\}$, we have $q=n-1$, and it follows from \eqref{eq:identity:(n-1)!:a} that
$$
\sum_{I\in \Pc_0(\kappa)}(|I|-1)!(|I^c|-1)!  =\frac{(n-1)!}{2!(n-3)!}\cdot2\cdot(n-3)!=(n-1)!
$$
Thus \eqref{eq:identity:(n-1)!} is true in this case.

\medskip

Assume now that \eqref{eq:identity:(n-1)!}  holds for some $\kappa'=(k'_1,\dots,k'_n)$ where $\kappa'$ has exactly $p-1$ positive entries.
We can assume that $k_i>0$ if and only if  $i\in \{1,\dots,p\}$, and that $k'_i>0$ if and only if $i\in\{1,\dots,p-1\}$.
For any $J\subset\{1,\dots,p\}$, let us write
$$
s(J):=\sum_{i\in J}k_i \quad \text{ and } \quad s'(J):=\sum_{i\in J}k'_i.
$$
Apply Lemma~\ref{lm:identity:(n-1)!:a} to $\kappa'$ with  the set $P=\{1,\dots,p\}$, we get
\begin{equation}\label{eq:identity:(n-1)!:b}
(n-1)!= q!\sum_{\substack{J\subset P}}\frac{(s'(J)+|J|+1)!}{(s'(J)+2)!}\cdot \frac{(s'(J^c)+|J^c|+1)!}{(s'(J^c)+2)!}=\frac{2q!(n-3)!}{2!(q-2)!}+F_{p,2,2}(k'_1,\dots,k'_p).
\end{equation}
where $F_{p,2,2}$ is the function defined in Theorem~\ref{th:sym:poly}.
We now remark that
$$
k'_1+\dots+k'_p=k_1+\dots+k_p=q-4,
$$
where $q=n-p$. Therefore, Theorem~\ref{th:sym:poly} implies that $F_{p,2,2}(k'_1,\dots,k'_p)=F_{p,2,2}(k_1,\dots,k_p)$.
It follows  that
\begin{align*}
(n-1)! &= \sum_{I\in \Pc_0(\kappa')}(|I|-1)!(|I^c|-1)! \quad (\text{ by assumption})\\
  &=\frac{2q!(n-3)!}{2!(q-2)!}+ q!F_{p,2,2}(k'_1,\dots,k'_p) \quad (\hbox{ by Lemma \ref{lm:identity:(n-1)!:a}})\\
&= \frac{2q!(n-3)!}{2!(q-2)!}+ q!F_{p,2,2}(k_1,\dots,k_p) \quad (\hbox{ by Theorem~\ref{th:sym:poly}})\\
&= \sum_{I\in \Pc_0(\kappa)}(|I|-1)!(|I^c|-1)! \quad (\hbox{ by Lemma~\ref{lm:identity:(n-1)!:a}}).
\end{align*}
Thus \eqref{eq:identity:(n-1)!} holds for all $\kappa$ satisfying $(*)$, and the theorem follows.
\end{proof}
\subsection{Comparison with Kontsevich's formula}\label{sec:compare:Kont:form}
Let $V_{MV}(\kappa)$ denote the total volume of $\Pb\Omega^2\Mod_{0,n}(\kappa)$ with respect to the Masur-Veech volume $d\vol^{MV}_1$.
Here the volume forms $d\vol^{MV}$ and $d\vol_1^{MV}$ are normalized as in \cite[\textsection 4.1]{AEZ16}.
In particular,  we have (see \cite[(4.5)]{AEZ16})
\begin{equation}\label{eq:MV:vol:norm:AEZ}
V_{MV}(\kappa)=\dim_\R\Omega^2\Mod_{0,n}(\kappa)\cdot\vol^{MV}(\Omega^2_{1}\Mod_{0,n}(\kappa))=2(n-2)\cdot\vol^{MV}(\Omega^2_1\Mod_{0,n}(\kappa))
\end{equation}
where $\Omega^2_1\Mod_{0,n}(\kappa)$ is the set of $(X,q) \in \Omega^2\Mod_{0,n}(\kappa)$ such that the area of $(X,|q|)$ is at most $1/2$.

\begin{Proposition}\label{prop:compare:Kont:form}
Assume that $\kappa=(k_1,\dots,k_n)$ satisfies $k_1+\dots+k_n=-4$ and all the $k_i$'s are odd. Then the volume forms $d\vol^{MV}$ and $d\vol$ satisfy
\begin{equation}\label{eq:ratio:vol:forms}
d\vol^{MV}=(-1)^{\frac{n-2}{2}}\cdot 2^{n-2}\cdot d\vol.
\end{equation}
As a consequence, we have
\begin{equation}\label{eq:Kont:form:proof}
V_{MV}(\kappa)= 2\cdot\pi^{n-2}\cdot \prod_{i=1}^n\frac{k_i!!}{(k_i+1)!!}.
\end{equation}
\end{Proposition}
\begin{Remark}\label{rk:Kont:form:AEZ}
Formula \eqref{eq:Kont:form:proof} was conjectured by Kontsevich and proved by Athreya-Eskin-Zorich in \cite{AEZ16}.
\end{Remark}
\begin{proof}
We first prove \eqref{eq:ratio:vol:forms}. Consider a point $(\CP^1,x_1,\dots,x_n,q)\in \Omega^2\Mod_{0,n}(\kappa)$, (here $x_1,\dots,x_n$ are $n$ marked points on $\CP^1$, and $q$ is a quadratic differential having a zero of order $k_i$ at $x_i$). Remark that in this case  $n$ must be even and $n\geq 4$.

The canonical double cover of $(\CP^1,x_1,\dots,x_n,q)$ is a hyperelliptic Riemann surface $\hX$ of genus $g$ with $n=2g+2$.
Since all the $k_i$'s are odd, the preimage of $x_i$ consists of a single point in $\hX$, which will be denoted by $\hx_i$.
Let $\tau$ denote the hyperelliptic involution of $\hX$. Note that $\{\hx_1,\dots,\hx_n\}$ is the set of fixed points of $\tau$.
The space $\Omega^2\Mod_{0,n}(\kappa)$ is locally modeled by the eigenspace $H^1(\hX,\{\hx_1,\dots,\hx_n\},\C)^-$ associated to the eigenvalue $-1$ of the action of $\tau$ on  $H^1(\hX,\{\hx_1,\dots,\hx_n\},\C)$. Since $\hX$ is hyperelliptic, we have $\tau^*\eta=-\eta$ for all $\eta \in H^1(\hX,\C)$, which means that $H^1(\hX,\C)=H^1(\hX,\C)^-$. It is a well known fact that in this case the projection $\pp: H^1(\hX,\{\hx_1,\dots,\hx_n\},\C)^- \to H^1(\hX,\C)^-=H^1(\hX,\C)$ is an isomorphism. Thus we can locally identified $\Omega^2\Mod_{0,n}(\kappa)$ with $H^1(\hX,\C)$.

Let $\{a_1,\dots,a_g,b_1,\dots,b_g\}$ be a symplectic basis of $H_1(\hX,\Z)$. Using this basis, we can identify $H^1(\hX,\C)$ with $\C^{2g}=\C^{n-2}$.
The Masur-Veech volume form is normalized such that the lattice $(\Z\oplus \imath\Z)^{n-2}$ has covolume $1$ (see \cite[\textsection 4.1]{AEZ16}).
This means that $d\vol^{MV}$ is the Lebesgue measure of $\C^{n-2}$.
In the standard coordinates $(z_1,\dots,z_{n-2})$ of $\C^{n-2}$, we can write
$$
d\vol^{MV}=\left(\frac{\imath}{2}\right)^{n-2}dz_1\wedge d\bar{z}_1\wedge\dots\wedge dz_{n-2}\wedge d\bar{z}_{n-2}.
$$
In this context, the volume form $d\vol$ is defined as follows: since the intersection form on $H^1(\hX,\C)$ is given by a non-degenerate Hermitian matrix $H$ of signature $(g,g)$, $\vartheta:=\mathrm{Im}(H)$ is a symplectic form on $H^1(\hX,\C)$. Therefore, $\vartheta^{2g}$ is a volume form on $H^1(\hX,\C)$, and $d\vol$ is defined to be the volume form on $\Omega^2\Mod_{0,n}(\kappa)$ that is induced by $\vartheta^{2g}/(2g)!$ (recall that $\Omega^2\Mod_{0,n}(\kappa)$ is locally identified with $H^1(\hX,\C)$).
Since the matrix of the intersection form in the basis $\{a_1,\dots,a_g,b_1,\dots,b_g\}$ is $H=\frac{\imath}{2}\cdot\left(
\begin{array}{cc}
0 & I_g \\
-I_g & 0
\end{array}
\right)
$, we have
$$
d\vol =\frac{\vartheta^{2g}}{(2g)!} = \det(H)\cdot \left(\frac{\imath}{2}\right)^{n-2}dz_1\wedge d\bar{z}_1\wedge\dots\wedge dz_{n-2}\wedge d\bar{z}_{n-2} = \frac{(-1)^{\frac{n-2}{2}}}{2^{n-2}}\cdot d\vol^{MV}
$$
and \eqref{eq:ratio:vol:forms} follows.
We now have
\begin{align*}
V_{MV}(\kappa) &= 2\cdot(n-2)\cdot\vol^{MV}(\Omega^2_1\Mod_{0,n}(\kappa)) \quad (\hbox{ by the normalization \eqref{eq:MV:vol:norm:AEZ}})\\
               &=2\cdot(n-2)\cdot(-1)^{\frac{n-2}{2}}\cdot 2^{n-2}\cdot \vol(\Omega^2_1\Mod_{0,n}(\kappa)) \quad (\hbox{by \eqref{eq:ratio:vol:forms}})\\
               &=2\cdot(n-2)\cdot(-1)^{\frac{n-2}{2}}\cdot 2^{n-2}\cdot\frac{1}{2}\cdot \vol_1(\Pb\Omega^2\Mod_{0,n}(\kappa)) \quad (\hbox{ by  the relation \eqref{eq:rel:vols:proj:n:area:1}})\\
               &=2\cdot (n-2)\cdot(-1)^{\frac{n-2}{2}}\cdot 2^{n-3} \cdot (-1)^{\frac{n-2}{2}}\cdot\frac{\pi^{n-2}}{2^{n-3}(n-2)}\cdot \prod_{i=1}^n\frac{k_i!!}{(k_i+1)!!}  \quad (\hbox{ by Theorem~\ref{th:vol:quad:diff:stra}})\\
               &=2\cdot\pi^{n-2}\cdot\prod_{i=1}^n\frac{k_i!!}{(k_i+1)!!}.
\end{align*}
The proposition is then proved.
\end{proof}


\section{Volume of moduli spaces of flat surfaces}\label{sec:vol:mod:flat:surf}
Throughout this section, $\mu=(\mu_1,\dots,\mu_n)$ will be a vector in $(\R_{<1})^n$ which satisfies $\mu_1+\dots+\mu_n=2$. Recall that $\CC(\mu)$ is the space of flat surfaces of genus $0$, considered up to a rescaling, which have exactly $n$ singular points with cones angles given by $\theta_i:=2\pi(1-\mu_i), \; i=1,\dots,n$.

\subsection{Local coordinates and volume form}
We first briefly recall the construction of local charts and the volume form $\vol_\mu$ on $\CC(\mu)$. The content of this paragraph is well known, we refer to  \cite{Th98} or \cite{Veech:FS} for a more detailed account on the matters.

Consider a flat surface $M$ representing an element of $\CC(\mu)$. Let $\Sig:=\{s_1,\dots,s_n\}$ denote the set of singularities of $M$, where the cone angle at $s_i$ is $\theta_i$.
We pick  a geodesic triangulation $\T$ of $M$ with $\Sig$ being the vertex set. Using the developing map, one can associate to each edge $e$ of $\T$ (with a chosen orientation) a complex number $z(e)$. The collection of complex numbers associated to the edges of $\T$  gives us a vector $Z \in \C^{N_1}$, where $N_1$ is the number of edges of $\T$. The coordinates of $Z$ satisfy some obvious linear relations relative to boundaries of the triangles in $\T$ and the holonomy of the flat metric. This means that $Z$ is contained in a linear subspace $V_\mu$ of $\C^{N_1}$ of dimension $n-2$. Note that if $M'$ is another flat surface close enough to $M$ in $\CC(\mu)$ then the same construction applied to $M'$ yields a vector $Z'\in V_\mu$ close to $Z$.

There is a $N_1\times N_1$ Hermitian matrix $H$ independent of $\mu$ such that
\begin{equation}\label{eq:FS:area:herm:form}
\Aa(M)={}^t\ol{Z}\cdot H \cdot Z.
\end{equation}
In \cite{Veech:FS}, Veech showed that if none of the weights $\mu_i$ is an integer, then the restriction of $H$ to $V_\mu$, which will be denoted by $H_\mu$, is non-degenerate.
In this case, let $\vartheta_\mu$ denote the imaginary part of $H_\mu$. Then
$$
d\Vol_\mu:=\frac{\vartheta^{n-2}_\mu}{(n-2)!}
$$
is a volume form on $V_\mu$.

Let $V_\mu^+$ denote the set $\{Z \in V_\mu, \; {}^t\ol{Z}\cdot H \cdot Z >0\}$, and $\Pb V^+_\mu$ be the projection of $V_\mu^+$ in $\Pb V_\mu$.
We obtain a volume form $d\vol_\mu$ on $\Pb V^+_\mu$ from $d\Vol_\mu$
as follows: for every open subset $B \subset \Pb V^+_\mu$, let
$$
C_1(B):=\{Z \in V_\mu^+, \; \C\cdot Z \in B, {}^t\ol{Z}\cdot H\cdot Z < 1\}.
$$
Then $\vol_\mu(B):=\Vol_\mu(C_1(B))$.
Since  $\CC(\mu)$ is locally modeled by $\Pb V^+_\mu$, $d\vol_\mu$ provides us with a well defined volume form on $\CC(\mu)$. It has been shown by Thurston~\cite{Th98} and Veech~\cite{Veech:FS} that the total volume of $\CC(\mu)$ with respect to $d\vol_\mu$  is finite (see also \cite{Ng12} for an alternate proof).

In the case $\mu\in \left(\frac{1}{d}\Q\right)^n$, the space $\CC(\mu)$ is in fact a stratum of projectivized $d$-differentials in genus $0$, and we have $d\vol_\mu=d\vol_1$, where $\d\vol_1$ is the  volume form defined in \textsection~\ref{sec:vol:form:def}.


Recall that $\Lb_n$ is the set of $\mu=(\mu_1,\dots,\mu_n)\in (\R_{<1})^n$  such that $\mu_1+\dots+\mu_n=2$.
Let $\Lb^*_n$ be the set of $\mu\in\Lb_n$ such that none of the entries of $\mu$ is an integer.
We now show
\begin{Lemma}\label{lm:vol:funct:conti}
There is a continuous function $\AA_n: \Lb_n \to \R$ which satisfies
\begin{itemize}
\item for all $\mu\in \Lb_n^*$, $\vol_\mu(\CC(\mu))=\frac{(-1)^{n-3}\pi^{n-2}}{(n-2)!}\AA_n(\mu)$,

\item if $\mu \in \Lb_n\setminus\Lb_n^*$, then $\AA_n(\mu)=0$,

\item whenever $\mu\in \Q^n\cap\Lb_n$, we have
\begin{equation}\label{eq:vol:mod:FS:rat:weight}
\AA_n(\mu)=\frac{1}{d^{n-3}}\cdot\Jc_n(\mu)=\frac{1}{d^{n-3}}\cdot(c_1^{n-3}(\hat{\Lc}_\mu)\cap[\blowupsp]),
\end{equation}
where $d$ is the smallest positive integer such that $d\cdot\mu_i\in \Z$, for all $i=1,\dots,n$, and $\kappa=d\cdot\mu \in \Z^n$.
\end{itemize}
\end{Lemma}
\begin{proof}[Sketch of proof]
Consider a point $\xx\in \Mod_{0,n}$ that parametrizes a pointed curve $(\CP^1,x_1,\dots,x_n)$.
For all $\mu \in \Lb_n$,  the metric $\frac{\lambda|dx|^2}{\prod_{i=1}^n|x-x_i|^{2\mu_i}}$, for some $\lambda\in \R_{>0}$, defines a flat surface $M$ in $\CC(\mu)$. Up to scalar multiplication, the coordinates of the vector $Z$ above are given by the integrals of the multivalued $1$-form $\frac{dx}{\prod_{i=1}^n(x-x_i)^{\mu_i}}$ along paths with endpoints in the set $\{x_1,\dots,x_n\}$.
It follows that the correspondence $M \mapsto \C\cdot Z\in \CP^{N_1-1}$ provides us with a continuous map $\Phi: U\times W \to  \CP^{N_1-1}$, where $U$ is a neighborhood of $\xx$ in $\Mod_{0,n}$, and $W$ is a neighborhood of $\mu$ in $\Lb_n$. For any fixed $\mu$, $\Phi_\mu:=\Phi(.,\mu): U \to \CP^{N_1-1}$ maps $U$ onto an open subset of $\Pb V_\mu^+$.

Let $\CP^{N_1-1}_+$ be the subset of $\CP^{N_1-1}$ consisting of lines $\C\cdot Z$ such that ${}^t\ol{Z}\cdot H\cdot Z >0$.
By definition,  $H$ provides us with a Hermitian metric on the tautological line bundle over $\CP^{N_1-1}_+$.
Let $\Theta^\#$ be the curvature form of this metric.
Note that for all $\mu \in W$, $\Phi_\mu(U) \subset \CP^{N_1-1}_+$.
Let $d\widetilde{\vol}_\mu$ denote the pullback of $\left(\frac{\imath}{2\pi}\cdot\Theta^\#\right)^{n-3}$ by $\Phi_\mu$. We then define
$$
\AA_n(\mu):=\int_{\Mod_{0,n}}d\widetilde{\vol}_\mu.
$$
Be construction, $\AA_n: \Lb_n \to \R$ is a continuous function.
Now, it follows from \cite[Lem. 6.1]{Ng22}  that if $\mu \in \Lb^*_n$, then
$$
d\vol_\mu=\frac{(-1)^{n-3}\pi^{n-2}}{(n-2)!}d\widetilde{\vol}_\mu,
$$
and
$$
\vol_\mu(\CC(\mu))=\int_{\Mod_{0,n}}d\vol_\mu=\frac{(-1)^{n-3}\pi^{n-2}}{(n-2)!}\int_{\Mod_{0,n}}d\widetilde{\vol}_\mu= \frac{(-1)^{n-3}\pi^{n-2}}{(n-2)!}\AA_n(\mu).
$$
If $\mu \in \Lb_n\setminus\Lb^*_n$, that is $\mu_i\in \Z$ for some $i \in  \{1,\dots,n\}$, then the restriction of $H$ to $V_\mu$ is degenerate. This implies that the restriction of $(\Theta^\#)^{n-3}$ to $\Pb(V_\mu)^+$ vanishes identically, and therefore $\AA_n(\mu)=0$.

Suppose now that $\mu \in \Q^n$, and $d$ is the smallest positive integer such that $\kappa:=d\cdot\mu\in \Z^n$. In this case every flat surface $M$ in $\CC(\mu)$ is defined by a $d$-differential form $(\CP^1,x_1,\dots,x_n,q) \in \stratesp$.
Assume that $\mu \in \Lb^*_n$ then $d\widetilde{\vol}_\mu$ satisfies $d\widetilde{\vol}_\mu = \frac{1}{d^{n-3}}\left(\frac{\imath}{2\pi}\Theta\right)^{n-3}$, where $\Theta$ is the curvature form of the Hodge norm on $\OO(-1)_{\Pb\Omega^d\Mod_{0,n}(\kappa)}$.
It then follows from the results of \cite{CMZ19} and \cite{Ng22} that
$$
\AA_n(\mu)= \int_{\Mod_{0,n}}d\widetilde{\vol}_\mu  = \frac{1}{d^{n-3}}\cdot\int_{\Pb\Omega^d\Mod_{0,n}(\kappa)}(\frac{\imath}{2\pi}\cdot\Theta)^{n-3}
         = \frac{1}{d^{n-3}}\cdot \big(c_1^{n-3}(\hat{\Lc}_\mu)\cdot[\blowupsp]\big)
         =\frac{1}{d^{n-3}}\cdot\Jc_n(\mu).
$$

\end{proof}

\subsection{Proof of Theorem~\ref{th:vol:funct}}\label{subsec:vol:funct:prf}
\begin{proof}
The function $\AA_n$ we look for is defined in Lemma~\ref{lm:vol:funct:conti}.
We already know that $\AA_n$ satisfies (b) and (c).
It remains to show that $\AA_n$ is piecewise polynomial with rational coefficients.

Recall that $\Pc$ is the set of partitions $\Sc=\{I_0,I_1\}$ of $\{1,\dots,n\}$ such that $\min\{|I_0|,|I_1|\}\geq 2$.
For each $\Sc\in \Pc$, let $\Wb_\Sc$ denote the intersection of $\Lb_n$ with the hyperplane defined by the equation $\mu(I_0)-\mu(I_1)=0$.
Define $\Wb:=\bigcup_{\Sc \in \Pc}\Wb_\Sc$.
Let $\Ub$ be a connected component  of $\Lb_n\setminus\Wb$.
By definition, the sign of $\mu(I_1)-\mu(I_0)$ is constant for all $\{I_0,I_1\}\in \Pc$ as $\mu$ varies in $\Ub$.
We will call $\Ub$ a {\em sign domain} of $\Lb_n$.

Clearly the closures of sign domains give a decomposition of $\Lb_n$ into cells.
We claim that the restriction of $\AA$ to each sign domain $\Ub$ is given by a polynomial with rational coefficients of degree at most $n-3$.

For $n=3$, we have $\AA_3(\mu)=1$ for all $\mu \in \Lb_3$, and $\Pc=\vide$. Thus our claim trivially holds.

\medskip

For $n=4$, $\Pc$ contains only 3 partitions. If $\Sc=\{\{i,j\},\{k,\ell\}\}\in \Pc$, then using the condition that $\mu_1+\dots+\mu_4=2$, we can write
$$
\mu_\Sc=\frac{1}{2}|\mu_i+\mu_j-(\mu_k+\mu_\ell)|.
$$
Since $ \AA_4(\mu)=\mathrm{deg}\hat{\Dc}_\mu/d=\mathrm{deg}\Dc_\mu/d$,  we have
\begin{align*}
  \AA_4(\mu) &=\frac{1}{2\cdot 3}\sum_{\Sc\in \Pc}(1-3\mu_\Sc) \\
   &= \frac{1}{2}-\frac{1}{4}\left(|\mu_1+\mu_2-\mu_3-\mu_4|+|\mu_1+\mu_3-\mu_2-\mu_4|+|\mu_1+\mu_4-\mu_2-\mu_3|\right).
\end{align*}
It is straightforward to check that $\AA_4$ satisfies the claim.

\medskip

We now consider a sign domain $\Ub$ in $\Lb_n$ for some $n\geq 5$.
Given $\mu\in \Lb_n$, we define the subsets $\Tc_{x}(\mu)$ with $x \in \{1a,1b, 2a, 2b\}$ as in \ref{sec:state:results}.
\begin{Claim}\label{clm:sign:dom:adm:parti}
For all $\mu \in \Ub$, $\Tc_{1b}(\mu)=\Tc_{2b}(\mu)=\vide$.
\end{Claim}
\begin{proof}
If $\{I_{00},I_{01},I_1\}$ is a partition in $\Tc_{1b}(\mu)$, then $\{I_{00},I_{01}\cup I_1\}$ is a partition in $\Pc$, and we have
$$
\mu(I_{00})-\mu(I_{01}\cup I_1)=0.
$$
But by definition, we must have $\mu(J_1)-\mu(J_0)\neq 0$ for all $\{J_0,J_1\} \in \Pc$.
Thus, $\Tc_{1b}(\mu)=\vide$. The proof of $\Tc_{2b}(\mu)=\vide$ follows from similar arguments.
\end{proof}

\begin{Claim}\label{clm:dom:same:partitions}
For all $\mu, \mu' \in \Ub$, we have $\Tc_x(\mu)=\Tc_x(\mu')$, for  $x\in \{1a,2a\}$.
\end{Claim}
\begin{proof}
Consider a partition $\Sc=\{I_0,I_1\}\in \Tc_{1a}(\mu)$. We have $\mu(I_1) > 1 > \mu(I_0)$, or equivalently $\mu(I_1)-\mu(I_0)>0$ (since $\mu(I_1)+\mu(I_0)=2$).  By the definition of $\Ub$ we also have $\mu'(I_1)-\mu'(I_0)>0$, which means that $\Sc\in \Tc_{1a}(\mu')$ as well.

Consider now a partition $\Sc=\{I_{0},I_{1}, I_2\} \in \Tc_{2a}(\mu)$. We have
$$
\mu(I_1)>\mu(I_0)+\mu(I_2)=\mu(I_0\cup I_2) \quad \text{ and } \quad \mu(I_2) > \mu(I_0)+\mu(I_1)=\mu(I_0\cup I_1).
$$
Since $\Sc_1:=\{I_1, I_0\cup I_2\}$ and $\Sc_2:=\{I_2, I_0\cup I_1\}$ are elements of $\Pc$, we also have
$$
\mu'(I_1)>\mu'(I_0)+\mu'(I_2) \quad \text{ and } \quad \mu'(I_2) > \mu'(I_0)+\mu'(I_1),
$$
which implies that $\Sc\in \Tc_{2a}(\mu')$.
\end{proof}

\begin{Claim}\label{clm:sign:dom:induct}
Let $\mu$ and $\mu'$ are two vectors in the sign domain $\Ub$.
For all $\Sc\in \Tc_{1a}(\mu)=\Tc_{1a}(\mu')$, let $\nu(\Sc)$ (resp. $\nu'(\Sc)$) be the  weight vector associated to $\mu$ (resp. to $\mu'$) defined in \ref{sec:state:results}. Then $\nu(\Sc)$ and $\nu'(\Sc)$ belongs to the same sign domain in $\Lb_{n-1}$.

Similarly, for all $\Sc\in \Tc_{2a}(\mu)$,  denote by $\nu_1(\Sc),\nu_2(\Sc)$ (resp. $\nu'_1(\Sc), \nu'_2(\Sc)$ the weight vectors associated to $\Sc$ that are constructed from $\mu$ (resp. from $\mu'$) as defined in \ref{sec:state:results}. Then $\nu_i(\Sc)$ and $\nu'_i(\Sc)$ belong to the same sign domain in $\Lb_{n_i+1}$, $i=1,2$.
\end{Claim}
\begin{proof}
We will only give the proof for the case $\Sc\in \Tc_{1a}(\mu)$.
Given $\Sc\in \Tc_{1a}(\mu)$, for simplicity we will write $\nu$ and $\nu'$ instead of $\nu(\Sc)$ and $\nu'(\Sc)$.
We can renumber the entries of $\mu$ and $\nu$ such that $\nu_i=\mu_i$, for $i=1,\dots,n-2$, and $\nu_{n-1}=\mu_{n-1}+\mu_n$.
Consider now a partition $\{J_0,J_1\}$ of $\{1,\dots,n-1\}$ such that $\min\{|J_0|,|J_1|\}\geq 2$.
Assume that $n-1\in J_1$. Define $\tilde{J}_0:=J_0$ and $\tilde{J}_1:=J_1\cup\{n\}$. Then $\{\tilde{J}_0,\tilde{J}_1\}$ is a partition of $\{1,\dots,n\}$ which belongs to $\Pc$. By definition, we have $\nu(J_i)=\mu(\tilde{J}_i)$ and $\nu'(J_i)=\mu'(\tilde{J}_i), \; i=0,1$. In particular
$$
\nu(J_1)-\nu(J_0)=\mu(\tilde{J}_1)-\mu(\tilde{J}_0) \quad \text{ and } \nu'(J_1)-\nu'(J_0)=\mu'(\tilde{J}_1)-\mu'(\tilde{J}_0).
$$
Since $\mu(\tilde{J}_1)-\mu(\tilde{J}_0)$ and $\mu'(\tilde{J}_1)-\mu'(\tilde{J}_0)$ have the same sign, so do $\nu(J_1)-\nu(J_0)$ and $\nu'(J_1)-\nu'(J_0)$. This means that $\nu$ and $\nu'$ belong to the same sign domain in $\Lb_{n-1}$.
\end{proof}

We can now conclude the proof by induction on $n$.
Define
$$
a^*_\Sc(\mu):=\left\{
\begin{array}{cl}
a_\Sc(\mu)/d & \text{ for all } \Sc \in \Tc_{1a}(\mu), \\
a_\Sc(\mu)/d^2 & \text{ for all } \Sc \in \Tc_{2a}(\mu).
\end{array}
\right.
$$
Note that $a^*_\Sc$ is a polynomial with rational coefficient in the variables $(\mu_1,\dots,\mu_n)$, which has degree $1$ if $\Sc\in \Tc_{1a}(\mu)$ and degree $2$ if $\Sc\in \Tc_{2a}(\mu)$.
For all $\mu\in \Q^n\cap \Ub$, since $\Tc_{1b}(\mu)=\Tc_{2b}(\mu)=\vide$, the recursive formula \eqref{eq:recursive:1} can be rewritten as
\begin{align}
\label{eq:recursive:vol} \AA_n(\mu) &= \sum_{\Sc\in\Tc_{1a}(\mu)}a^*_\Sc(\mu)\AA_{n-1}(\nu(\Sc)) -\sum_{\Sc\in \Tc_{2a}(\mu)}a^*_\Sc(\mu)\AA_{n_1(\Sc)+1}(\nu_1(\Sc))\AA_{n_2(\Sc)+1}(\nu_2(\Sc)).
\end{align}
Claim~\ref{clm:dom:same:partitions} implies that the  expression on the right hand side of \eqref{eq:recursive:vol} is the same for all $\mu\in \Ub\cap\Q^n$.
The induction hypothesis and Claim \ref{clm:sign:dom:induct} then imply that each summand on the  right hand side of \eqref{eq:recursive:vol} is given by a polynomial function on $\Ub$ with rational coefficients and degree at most $n-3$. By the continuity of $\AA_n$ (c.f. Lemma~\ref{lm:vol:funct:conti}) we get the desired conclusion.
\end{proof}

\appendix

\section{A result on symmetric polynomials}\label{sec:sym:poly:append}
\begin{Theorem}\label{th:sym:poly}
Let $a$ be a real number. For all $k\in \Z_{>0}$ define
$$
P_{k,a}(X):=\prod_{i=1}^k(X+a+i).
$$
By convention, we set $P_{0,a}(X):=1$ for all $a\in \R$.
Given $a,b\in \R$, and $n \in \Z, n \geq 2$, define
\begin{eqnarray*}
F_{n,a,b}(X_1,\dots,X_n) &:= & \sum_{\substack{I \subset \{1,\dots,n\},\\ I\neq \varnothing, I \neq \{1,\dots,n\}}} P_{|I|-1,a}(\sum_{i\in I}X_i)\cdot P_{|I^c|-1,b}(\sum_{i\in I^c}X_i)\\
& = & \sum_{\substack{I \subset \{1,\dots,n\},\\ I\neq \varnothing, I \neq \{1,\dots,n\}}} \left((\sum_{i\in I}X_i+a+1)\dots(\sum_{i\in I}X_i+a+|I|-1)\cdot \right.\\
&  & \left. \cdot (\sum_{i\in I^c}X_i+b+1)\dots(\sum_{i\in I^c}X_i+b+|I^c|-1)\right).
\end{eqnarray*}
Then $F_{n,a,b}(X_1,\dots,X_n)$ depends only on the sum $X_1+\dots+X_n$, that is $F_{n,a,b}(X_1,\dots,X_n)\in \R[\sum_{i=1}^nX_i]$.
\end{Theorem}
\begin{proof}
Note that $F_{n,a,b}(X_1,\dots,X_n)$ is a symmetric polynomial of degree $n-2$ in $(X_1,\dots,X_n)$.
We have
$$
F_{2,a,b}(X_1,X_2)= 2,
$$
and
$$
F_{3,a,b}(X_1+X_2+X_3)=4(X_1+X_2+X_3)+3(a+b+2).
$$
Assume that the conclusion holds for some $n\geq 3$.
For all $i\in \{1,\dots,n+1\}$,  define
$$
G^{(i)}_{n+1,a,b}(X_1,\dots,\widehat{X}_i,\dots,X_{n+1}):=F_{n+1,a,b}(X_1,\dots,\underset{i}{\underbrace{0}},\dots,X_{n+1}).
$$
Consider
$$
G^{(n+1)}_{n+1,a,b}(X_1,\dots,X_n):=F_{n+1,a,b}(X_1,\dots,X_n,0).
$$
We will show that $G^{(n+1)}_{n+1,a,b}\in \R[\sum_{i=1}^n X_i]$.
Let  $J$ be a subset of $\{1,\dots,n\}$ such that $1 < |J| < n$. Denote by $J^c$ the complement of $J$ in $\{1,\dots,n\}$.
Consider $J$ as a subset of $\{1,\dots,n+1\}$.
The term in $G^{(n+1)}_{n+1,a,b}$ associated to $J$ is
$$
P_{|J|-1,a}(\sum_{i\in J}X_i)\cdot P_{|J^c|,b}(\sum_{i\in J^c}X_i)= (\sum_{i\in J^c}X_i+b+|J^c|)\cdot P_{|J|-1,a}(\sum_{i\in J}X_i)\cdot P_{|J^c|-1,b}(\sum_{i\in J^c}X_i)
$$
Let $\hat{J}=J\sqcup\{n+1\}$. Then the term in $G^{(n+1)}_{n+1,a,b}$ corresponding to $\hat{J}$ is equal to
$$
P_{|J|,a}(\sum_{i\in J}X_i)\cdot P_{|J^c|-1,b}(\sum_{i\in J^c}X_i)= (\sum_{i\in J^c}X_i+a+|J|)\cdot P_{|J|-1,a}(\sum_{i\in J}X_i)\cdot P_{|J^c|-1,b}(\sum_{i\in J^c}X_i)
$$
Thus the sum of the terms in $G^{(n+1)}_{n+1,a,b}$ corresponding to $J$ and $\hat{J}$ is equal to
$$
(\sum_{i=1}^nX_i+a+b+n)\cdot P_{|J|-1,a}(\sum_{i\in J}X_i)\cdot P_{|J^c|-1,b}(\sum_{i\in J^c}X_i)
$$
Consider now  a subset $I$ of $\{1,\dots,n+1\}$ such that $I\neq \vide$ and $I \neq \{1,\dots,n+1\}$.
If $I=\{n+1\}$, then the term corresponding to $I$ in $G^{(n+1)}_{n+1,a,b}$ is
$$
P_{n-1,b}(\sum_{i=1}^n X_i) \in \R[\sum_{i=1}^nX_i],
$$
and if $I=\{1,\dots,n\}$, then the term corresponding to $I$ is
$$
P_{n-1,a}(X_1+\dots+X_n) \in \R[X_1+\dots+X_n].
$$
Otherwise, there is a unique subset $J$ of $\{1,\dots,n\}$ such that $1 < |J| < n$ and either $I=J$ of $I=J\sqcup\{n+1\}$.
Therefore,
$$
G^{(n+1)}_{n+1,a,b}(X_1,\dots,X_n)=P_{n-1,a}(\sum_{i=1}^nX_i)+P_{n-1,b}(\sum_{i=1}^nX_i)+ (\sum_{i=1}^n X_i+a+b+n)\cdot F_{n,a,b}(X_1,\dots,X_n).
$$
By assumption $F_{n,a,b}(X_1,\dots,X_n)\in \R[\sum_{i=1}^nX_i]$, thus we have $G^{(n+1)}_{n+1,a,b}(X_1,\dots,X_n)\in \R[\sum_{i=1}^nX_i]$.
By the same argument, we also have $G^{(j)}_{n+1,a,b} \in \R[\sum_{\substack{1\leq i\leq n+1, i\neq j}}X_i]$
for all $j\in \{1,\dots,n\}$.

For $d\in \{1,\dots,n-1\}$, let $F^{(d)}_{n+1,a,b}$ be the degree $d$ component of $F_{n+1,a,b}$. We can write
$$
F^{(d)}_{n+1,a,b}=a_d(X_1+\dots+X_{n+1})^d+R_d(X_1,\dots,X_{n+1}),
$$
where $R_d$ is a homogeneous symmetric polynomial of degree $d$ that does not contain any term of the form $\lambda\cdot X_i^d$.
We now observe that  $F^{(d)}_{n+1,a,b}(X_1,\dots,X_n,0)$ is the degree $d$ component of $G^{(n+1)}_{n+1,a,b}$.
Since $G^{(n+1)}_{n+1,a,b} \in \R[\sum_{i=1}^nX_i]$, we must have
$$
F^{(d)}_{n+1,a,b}(X_1,\dots,X_n,0)=a_d(X_1+\dots+X_{n})^d
$$
and
$$
R_d(X_1,\dots,X_n,0)=0.
$$
This means that every term of $R_d$ contains $X_{n+1}$.
The same argument applied to $G^{(i)}_{n+1,a,b}$ implies that all the terms of $R_d$ contain $X_1\cdots X_{n+1}$.
But the degree of $R_d$ is $d \leq n-1$. Thus we must have $R_d\equiv 0$. It follows that
$$
F^{(d)}_{n+1,a,b}(X_1,\dots,X_n,X_{n+1})=a_d(X_1+\dots+X_{n+1})^d
$$
and hence $F_{n+1,a,b} \in \R[\sum_{i=1}^{n+1}X_i]$.
\end{proof}


\section{Computations of Masur-Veech volumes in low dimension \\  {\small by Vincent Koziarz and Duc-Manh Nguyen}}\label{sec:low:dim:calcul}
In this section, we compute the self-intersection number $\hat{\Dc}_\mu^{n-3}$, where $\hat{\Dc}_\mu$ is defined in \eqref{eq:tauto:ln:bdl:expr}, and the corresponding Masur-Veech volume of $\Pb\stratesp$, for $n=4,5$,  $d\in \{3,4,6\}$.
It is well known that  the Masur-Veech volume of $\Pb\stratesp$ gives the asymptotics of the number of tilings of the sphere by triangles and squares with prescribed constraints at some vertices (see for instance \cite{Th98, ES18, Engel-I, KN20}).
It is worth noticing that our calculations also show that $\hat{\Dc}_\mu^{n-3}=0$ if there exists $i\in \{1,\dots,n\}$ such that $\mu_i\in \Z_{\leq 0}$, in accordance with Theorem~\ref{th:recursive:1}.

To compute the ratio $\frac{d\vol^{MV}}{d\vol}$, our strategy goes as follows: we first observe that it is enough to compute this ratio locally near any point in $\stratesp$. Recall that an element $(\CP^1,x_1,\dots,x_n,q)\in \stratesp$ defines a flat metric with conical singularities on the sphere, where the cone angle at $x_i$ is  $\alpha_i:=2\pi(1+\frac{k_i}{d})$.
To simplify the discussion, let us assume that there is at most one angle in $\{\alpha_1,\dots,\alpha_n\}$ that is greater than $2\pi$. Thus, we can assume that $0< \alpha_i< 2\pi$ for $i=1,\dots,n-1$. We construct a flat surface in $\stratesp$ as follows: let $P$ be a $2(n-1)$-gon  in the plane whose vertices  are denoted by $p_1,\dots,p_{2(n-1)}$ in the counterclockwise ordering. We suppose that $P$ satisfies the following
\begin{itemize}
\item the sides $\ol{p_{2i-1}p_{2i}}$ and $\ol{p_{2i}p_{2i+1}}$ have the same length, and

\item the interior angle at $p_{2i}$ is $\alpha_i$,
\end{itemize}
for $i=1,\dots,n-1$, (by convention $p_{2n-1}=p_1$).
Gluing $\ol{p_{2i-1}p_{2i}}$ and $\ol{p_{2i}p_{2i+1}}$ together, we obtain a flat surface $M$ homeomorphic to the sphere with $n$ conical singularities $x_1,\dots,x_n$, where $x_i$ corresponds to the vertex $p_{2i}$ of $P$ for $i=1,\dots,n-1$, and $x_n$ is the  identification of $p_1,p_3,\dots,p_{2n-3}$. The cone angle at $x_i$ is clearly $\alpha_i$ for $i=1,\dots,n-1$. Thus the cone angle at $x_n$ must be $\alpha_n$. Hence $M$ is an element of $\stratesp$.

Let $z_i$ be the complex number associated to the vector $\overrightarrow{p_{2i}p_{2i-1}}$ for $i=1,\dots,n-1$. Then the complex number associated to $\overrightarrow{p_{2i}p_{2i+1}}$ is $e^{\imath\alpha_i}z_i$. It follows that the numbers $(z_1,\dots,z_{n-1})$ satisfy
\begin{equation}\label{eq:loc:ch:by:polygon}
(1-e^{\imath\alpha_1})z_1+\dots+(1-e^{\imath\alpha_{n-1}})z_{n-1}=0.
\end{equation}
A neighborhood of $M$ in $\stratesp$ can be identified with an open subset of the hyperplane $V$  defined by \eqref{eq:loc:ch:by:polygon} in $\C^{n-1}$. Since $e^{\imath\alpha_i} \neq 1$ for all $i=1,\dots,n-1$, we can identify this hyperplane with $\C^{n-2}$ via the natural projection $\C^{n-1} \to \C^{n-2}, \; (z_1,\dots,z_{n-1}) \mapsto (z_1,\dots,z_{n-2})$.
In the coordinates $(z_1,\dots,z_{n-2})$, the area of $M$ is given by a Hermitian matrix $H$, that is
$$
\Aa(M)=(\bar{z}_1,\dots,\bar{z}_{n-2})\cdot H \cdot {}^t(z_1,\dots,z_{n-2}).
$$
Hence $d\vol=\det H \cdot d\lambda_{2(n-2)}$, where $\lambda_{2(n-2)}$ is the Lebesgue measure on $\C^{n-2}$.

In this setting, the Masur-Veech volume on $\stratesp$ is the unique (real) $(n-2,n-2)$-form on $V$ such that the lattice $\Lambda:=V\cap (\Z\oplus \zeta\Z)^{n-1}$ has covolume $1$, where $\zeta=\imath$ if $d=4$, and $\zeta=e^{\frac{2\pi\imath}{3}}$ if $d\in \{3,6\}$.  Let $\Lambda'$ be the projection of $\Lambda$ in $\C^{n-2}$. Then $\Lambda'$ is a sublattice of $(\Z\oplus\zeta\Z)^{n-2}$. Let $m$ be the index of $\Lambda'$ in $(\Z\oplus\zeta\Z)^{n-2}$. Since $\Lambda'$ has covolume $m\cdot\imag(\zeta)^{n-2}$ with respect to the Lebesgue measure, we must have
$$
\frac{d\vol^{MV}}{d\lambda_{2(n-2)}}=\frac{1}{m\cdot\imag(\zeta)^{n-2}}.
$$
As a consequence, we get
\begin{equation}\label{eq:form:ratio:of:vols}
\frac{d\vol^{MV}}{d\vol}=\frac{1}{m\cdot \imag(\zeta)^{n-2}\cdot \det H}.
\end{equation}

In the cases where there are more than one angle in $\{\alpha_1,\dots,\alpha_n\}$ that are greater than $2\pi$, we construct a  surface in $\stratesp$ from several polygons, and compute the ratio $d\vol^{MV}/d\vol$ by the same method.

\subsection{Case $n=4$}\label{subsec:cal:inters:n4}
In this case, $\blowupsp$ is isomorphic to $\ol{\Mod}_{0,4}\simeq \CP^1$, and $\hat{\Dcal}_\mu=\Dcal_\mu$. Thus, the self-intersection number is simply the degree of $\Dcal_\mu$.
The values of the ratio $d\vol^{MV}/d\vol$ and the Masur-Veech volume of $\Pb\stratesp$ are recorded in Table~\ref{table:vol:n4} (the values on the last column are the product of the ones on the third and the fourth columns with $-\frac{\pi^2}{2\cdot d}$).

\begin{table}[h!]
\centering

\begin{tabular}{|  p{0.5cm}|  p{3 cm}|  p{2 cm} | p{2 cm} | p{2 cm} |}
\hline
\centering $d$ & \centering $(-k_i)$ & \centering $\frac1{d}\deg(\hat{\cal D}_\mu)$ & \centering $d\vol^{MV}/d\vol$ & \centering M-V volume \tabularnewline
\hline
\end{tabular}
\begin{tabular}{|  p{0.5cm}|  p{3 cm}|  p{2 cm} | p{2 cm} | p{2 cm} |}
\hline
\centering $2$ &$(1,1,1,1)$ & \centering $1/2$ & \centering $-1$  & \centering $\pi^2/8$ \tabularnewline
\hline

\centering $3$ &$(2,2,1,1)$ & \centering $1/3$ & \centering $-16/9 $ & \centering $8\pi^2/81$ \tabularnewline
\hline

\centering $4$ &$(3,2,2,1)$ & \centering $1/4$ & \centering $-1$ & \centering $\pi^2/32$\tabularnewline
\hline
\centering $4$ &$(3,3,1,1)$ & \centering $1/4$ & \centering $-2$ & \centering $\pi^2/16$ \tabularnewline
\hline
\centering $4$ &$(3,3,3,-1)$ & \centering $-1/4$ & \centering $2$& \centering $\pi^2/16$ \tabularnewline
\hline

\centering $6$ &$(4,3,3,2)$ & \centering $1/3$ & \centering  $-4/9$ & \centering $\pi^2/81$ \tabularnewline
\hline
\centering $6$ &$(4,4,3,1)$ & \centering $1/6$ & \centering $-8/9$ & \centering $\pi^2/81$ \tabularnewline
\hline
\centering $6$ &$(5,3,2,2)$ & \centering $1/6$ & \centering $-8/9$ & \centering $\pi^2/81$ \tabularnewline
\hline
\centering $6$ &$(5,3,3,1)$ & \centering $1/6$ & \centering $-4/3$ & \centering $\pi^2/54$ \tabularnewline
\hline
\centering $6$ &$(5,4,2,1)$ & \centering $1/6$ & \centering $-16/9$ & \centering $2\pi^2/81$ \tabularnewline
\hline
\centering $6$ &$(5,4,4,-1)$ & \centering $-1/6$ & \centering $16/9$ & \centering $2\pi^2/81$ \tabularnewline
\hline
\centering $6$ &$(5,5,1,1)$ & \centering $1/6$ & \centering $-16/3$ & \centering $2\pi^2/27$ \tabularnewline
\hline
\centering $6$ &$(5,5,3,-1)$ & \centering $-1/6$ & \centering $8/3$ & \centering $\pi^2/27$ \tabularnewline
\hline
\centering $6$ &$(5,5,4,-2)$ & \centering $-1/3$ & \centering $16/9$ & \centering $4\pi^2/81$ \tabularnewline
\hline
\centering $6$ &$(5,5,5,-3)$ & \centering $-1/2$ & \centering $8/3$ & \centering $\pi^2/9$ \tabularnewline
\hline
\end{tabular}
\caption{Volume of $\Pb\stratesp$ when $n=4$.}
\label{table:vol:n4}
\end{table}

\subsection{Case $n=5$}\label{subsec:cal:inters:n5}
In this case  the exceptional divisor $\Ecal$ is non-trivial if and only if there exists a partition $\{I_0,I_1,I_2\}$ of $\{1,\dots,5\}$ such that $\sum_{i\in I_1} \mu_i>1$ and $\sum_{i\in I_2}\mu_i>1$ (see \cite[Th. 1.2(a)]{Ng23:a}).
This implies that $|I_0|=1$ and $|I_1|=|I_2|=2$.
Thus the stratum of $\ol{\Mod}_{0,5}$ corresponding to this partition is just one point.
Let $E$ be the  Weil divisor associated to the Cartier divisor $\Ecal$.
Then each component of $E$ maps to a single point in $\ol{\Mod}_{0,5}$.
Therefore, we have
$$
\hat{p}^*\Dcal_\mu\cdot E= \Dc_\mu\cdot \hat{p}_* E=0.
$$
Hence
\begin{equation}\label{eq:self:inters:n5}
(\hat p^*\Dcal_\mu+\Ecal)^2=\Dcal^2_\mu+\Ecal^2.
\end{equation}
Since we have an explicit expression of $\Dcal_\mu$ and the intersection of boundary divisors in $\ol{\Mod}_{0,n}$ is well understood, the computation of $\Dcal_\mu^2$ is straightforward.
It remains to compute $\Ecal^2$.

To simplify the discussion, let us assume that $E$ has  a unique irreducible component associated to a partition $\{I_0,I_1,I_2\}$ as above.
Define $r_i:=d(\sum_{k\in I_i} \mu_k-1), \; i=1,2$.
By assumption, we have $r_i \in \Z_{>0}$.

By construction, a neighborhood of $E$ in $\widehat{\Mod}_{0,5}(\kappa)$ can be described as $\widehat{\Ucal}=\{(t_1,t_2,[u:v]) \in \Ucal\times\Pb^1, \; vt^{r_1}_1=ut_2^{r_2}\}$, where $\Ucal$ is a neighborhood of $(0,0)$ in $\C^2$.
Define $U=\{(t_1,t_2,[u:v]) \in \widehat{\Ucal}, \; v \neq 0\}$ and $V=\{(t_1,t_2,[u:v])\in \widehat{\Ucal}, \; u\neq 0\}$.
We have
$$
U\simeq \{(t_1,t_2,u) \in \Ucal\times \C, \; t_1^{r_1}=ut_2^{r_2}\} \text{ and } V\simeq \{(t_1,t_2,v) \in \Ucal\times \C, \; vt_1^{r_1}=t_2^{r_2}\}.
$$
Note that $E\cap U$ and $E\cap V$ are both defined by $t_1=t_2=0$.
\begin{Lemma}\label{lm:inters:Weil:Cart:div}
We have $\Ecal\cdot E=-1$.
\end{Lemma}
\begin{proof}
 Since $E\simeq \Pb^1$, it is enough to check that the restriction of the line bundle associated to $\Ecal$ to $E$ is precisely the tautological line bundle $\Ocal_{\Pb^1}(-1)$.
 This is straightforward from the construction.
\end{proof}

We now show the following formula, which generalizes the usual self-intersection number of the exceptional divisor of the blow-up of a surface at one point.
\begin{Lemma}\label{lm:self:inters:nb:ex:div}
We have
\begin{equation}\label{eq:self:inters:ex:div}
\Ecal^2=-r_1r_2.
\end{equation}
\end{Lemma}

By Lemma~\ref{lm:inters:Weil:Cart:div}, it is enough to show that  the order of $\Ecal$ along $E$ is $r_1r_2$. This is infact a consequence of \cite[Th. 1.2(a)]{Ng23:a}. However we will give here below an alternative proof using complex analytic arguments.

\begin{proof}[Proof of Lemma~\ref{lm:self:inters:nb:ex:div} in the analytic setting]
Since the order of a function along a Weil divisor is only well defined for normal spaces, we first need to pass to the normalization of $\widehat{\Ucal}$.
Let $\tilde{U}$ (resp. $\tilde{V}$) be the normalization of $U$ (resp. $V$). Then the normalization $\tilde{\Ucal}$ of $\widehat{\Ucal}$ is obtained by gluing $\tilde{U}$ and $\tilde{V}$ over $U\cap V$.
Let $\tilde{\varphi}: \tilde{\Ucal} \to \widehat{\Ucal}$ denote the normalizing map.

Let $c:=\gcd(r_1,r_2)$ and $m_i:=r_i/c, \; i=1,2$. Denote by $G_{m_i}$ the cyclic group $\Z/m_i\Z$. It is a well known fact that $\tilde{U}$ is isomorphic to  the quotient $\C^2/G_{m_1}$, where the action of $G_{m_1}$ on $\C^2$ is generated by $(t,x) \mapsto (\zeta_1t,\zeta_1^{-m_2}x)$ (here $\zeta_1$ is a primitive $m_1$-th root of unity). The normalizing projection $\tilde{\varphi}_{|\tilde{U}}: \tilde{U} \to U$ is induced by  the map $\varphi_U: \C^2 \to U, \; (t,x) \mapsto (t_1,t_2,u)=(xt^{m_2},t^{m_1},x^{r_1})$.

Let $\tilde{E}$ be the pre-image of $E$ in $\tilde{\Ucal}$.
The intersection $\tilde{E}\cap \tilde{U}$ is identified with $(\{0\}\times\C)/G_{m_1}$.
Observe that the restriction of $\tilde{\varphi}$ to $\tilde{E}$ has degree $c$.
This is because $\widehat{\Ucal}$ has $c$ local branches through every point $\{(0,0)\}\times \C^*$.
Therefore, we have
$$
\tilde{E}\cdot \tilde{\varphi}^*\Ecal= c\cdot(E\cdot\Ecal)=-c.
$$
A neighborhood of a regular point in $\tilde{E}$ can be identified with a neighborhood in $\C^2$ of a point $(0,x)$ with $x\neq 0$.
In this neighborhood $\tilde{E}$ is defined by the equation $t=0$, while the Cartier divisor $\tilde{\varphi}^*\Ecal$ is  defined by $t_2^{r_2}=t^{r_2m_1}=t^{cm_1m_2}$.
Thus the order of $\tilde{\varphi}^*\Ecal$ along $\tilde{E}$ is $cm_1m_2$.
It follows that
$$
\tilde{\varphi}^*\Ecal\cdot \tilde{\varphi}^*\Ecal=cm_1m_2\cdot(\tilde{E}\cdot\tilde{\varphi}^*\Ecal)=-c^2m_1m_2=-r_1r_2.
$$
Since $\deg \tilde{\varphi}=1$, we get $\Ecal^2=-r_1r_2$.
\end{proof}

In Table~\ref{table:vols:n5}, we record the values of $\hat{\Dcal}_\mu^2$ in the case where $n=5$ and $d\in \{3,4,6\}$, together with the values of the Masur-Veech volume of the corresponding stratum of $d$-differentials. The values on the last column  are the product of the ones in the third and the fourth columns with $\frac{\pi^3}{3!d}$.

\begin{table}[h]
\centering

\begin{tabular}{|  p{0.5cm}|  p{3 cm}|  p{2 cm} | p{2 cm} | p{2 cm} |}
\hline
\centering $d$ & \centering $(-k_i)$ & \centering $\frac1{d^2}\hat{\cal D}_\mu^{2}$ & \centering $d\vol^{MV}/d\vol$ & M-V volume \tabularnewline
\hline
\end{tabular}
\begin{tabular}{|  p{0.5cm}|  p{3 cm}|  p{2 cm} | p{2 cm} | p{2 cm} |}
\hline
\centering $3$ &$(2,1,1,1,1)$ & \centering $1/9$ & \centering $64/27$ & \centering $32\pi^3/2187$ \tabularnewline
\hline
\centering $3$ &$(2,2,2,1,-1)$ & \centering $-2/9$ & \centering $-64/27$ & \centering $64\pi^3/2187$ \tabularnewline
\hline
\centering $3$ &$(2,2,2,2,-2)$ & \centering $-2/9$ & \centering $-64/27$ & \centering $64\pi^3/2187$ \tabularnewline

\hline
\centering $4$ &$(2,2,2,1,1)$ & \centering $1/8$ & \centering $1$ & \centering $\pi^3/192$ \tabularnewline
\hline
\centering $4$ &$(3,2,1,1,1)$ & \centering $1/16$ & \centering $2$& \centering $\pi^3/192$ \tabularnewline
\hline
\centering $4$ &$(3,2,2,2,-1)$ & \centering $-3/16$ & \centering $-1$ & \centering $\pi^3/128$ \tabularnewline
\hline
\centering $4$ &$(3,3,2,1,-1)$ & \centering $-1/8$ & \centering $-2$ & \centering $\pi^3/96$ \tabularnewline
\hline
\centering $4$ &$(3,3,2,2,-2)$ & \centering $-1/4$ & \centering $-1$ & \centering $\pi^3/96$  \tabularnewline
\hline
\centering $4$ &$(3,3,3,1,-2)$ & \centering $-1/4$ & \centering $-2$ & \centering $\pi^3/48$ \tabularnewline
\hline
\centering $4$ &$(3,3,3,2,-3)$ & \centering $-3/16$ & \centering $-2$ &  \centering $\pi^3/64$ \tabularnewline

\hline
\centering $6$ &$(3,3,2,2,2)$ & \centering $1/6$ & \centering $16/27$ & \centering $4\pi^3/243$ \tabularnewline
\hline
\centering $6$ &$(3,3,3,2,1)$ & \centering $1/9$ & \centering $8/9$ & \centering $2\pi^3/729$ \tabularnewline
\hline
\centering $6$ &$(4,3,2,2,1)$ & \centering $1/12$ & \centering $32/27$ & \centering $2\pi^3/729$ \tabularnewline
\hline
\centering $6$ &$(4,3,3,1,1)$ & \centering $1/18$ & \centering $16/9$ & \centering $2\pi^3/729$ \tabularnewline
\hline
\centering $6$ &$(4,3,3,3,-1)$ & \centering $-5/36$ & \centering $-8/9$ & \centering $5\pi^3/1458$ \tabularnewline
\hline
\centering $6$ &$(4,4,2,1,1)$ & \centering $1/18$ & \centering $64/27$ & \centering $8\pi^3/2187$\tabularnewline
\hline
\centering $6$ &$(4,4,3,2,-1)$ & \centering $-1/9$ & \centering $-32/27$ & \centering $8\pi^3/2187$ \tabularnewline
\hline
\centering $6$ &$(4,4,3,3,-2)$ & \centering $-2/9$ & \centering $ -16/27$ & \centering $8\pi^3/2187$ \tabularnewline
\hline
\centering $6$ &$(4,4,4,1,-1)$ & \centering $-1/18$ & \centering $-64/27$ & \centering $8\pi^3/2187$  \tabularnewline
\hline
\centering $6$ &$(4,4,4,3,-3)$ & \centering $-1/4$ & \centering $-16/9$ & \centering $\pi^3/81$ \tabularnewline
\hline
\centering $6$ &$(5,2,2,2,1)$ & \centering $1/36$ & \centering $64/27$ & \centering $4\pi^3/2187$ \tabularnewline
\hline
\centering $6$ &$(5,3,2,1,1)$ & \centering $1/36$ & \centering $32/9$ & \centering $2\pi^3/729$ \tabularnewline
\hline
\centering $6$ &$(5,3,3,2,-1)$ & \centering $-1/12$ & \centering $-16/9$ & \centering $\pi^3/243$ \tabularnewline
\hline
\centering $6$ &$(5,3,3,3,-2)$ & \centering $-2/9$ & \centering $-8/9$ & \centering $4\pi^3/729$ \tabularnewline
\hline
\centering $6$ &$(5,4,1,1,1)$ & \centering $1/36$ & \centering $64/9$ & \centering $4\pi^3/729$  \tabularnewline
\hline
\centering $6$ &$(5,4,2,2,-1)$ & \centering $-1/12$ & \centering $-64/27$ & \centering $4\pi^3/729$ \tabularnewline
\hline
\centering $6$ &$(5,4,3,1,-1)$ & \centering $-1/18$ & \centering $-32/9$ & \centering $4\pi^3/729$ \tabularnewline
\hline
\centering $6$ &$(5,4,3,2,-2)$ & \centering $-1/6$ & \centering $-32/27$ & \centering $4\pi^3/729$ \tabularnewline
\hline
\centering $6$ &$(5,4,3,3,-3)$ & \centering $-1/4$ & \centering $-8/9$ & \centering $\pi^3/162$ \tabularnewline
\hline
\centering $6$ &$(5,4,4,1,-2)$ & \centering $-1/9$ & \centering $-64/27$ & \centering $16\pi^3/2187$ \tabularnewline
\hline
\centering $6$ &$(5,4,4,2,-3)$ & \centering $-1/4$ & \centering $-32/27$ & \centering $2\pi^3/243$ \tabularnewline
\hline
\centering $6$ &$(5,4,4,3,-4)$ & \centering $-2/9$ & \centering $-32/27$ & \centering $16\pi^3/2187$ \tabularnewline
\hline
\centering $6$ &$(5,4,4,4,-5)$ & \centering $-5/36$ & \centering $-64/27$ & \centering $20\pi^3/2187$ \tabularnewline
\hline
\centering $6$ &$(5,5,2,1,-1)$ & \centering $-1/18$ & \centering $-64/9$ & \centering $8\pi^3/729$ \tabularnewline
\hline
\centering $6$ &$(5,5,2,2,-2)$ & \centering $-1/9$ & \centering $-64/27$ & \centering $16\pi^3/2187$  \tabularnewline
\hline
\centering $6$ &$(5,5,3,1,-2)$ & \centering $-1/9$ & \centering $-32/9$ & \centering $8\pi^3/729$ \tabularnewline
\hline
\centering $6$ &$(5,5,3,2,-3)$ & \centering $-1/6$ & \centering $-16/9$ & \centering $2\pi^3/243$ \tabularnewline
\hline
\centering $6$ &$(5,5,3,3,-4)$ & \centering $-2/9$ & \centering $-16/9$ & \centering $8\pi^3/729$ \tabularnewline
\hline
\centering $6$ &$(5,5,4,-1,-1)$ & \centering $1/18$ & \centering $64/9$ & \centering $ 8\pi^3/729$ \tabularnewline
\hline
\centering $6$ &$(5,5,4,1,-3)$ & \centering $-1/6$ & \centering $-32/9$ & \centering $4\pi^3/243$ \tabularnewline
\hline
\centering $6$ &$(5,5,4,2,-4)$ & \centering $-2/9$ & \centering $-64/27$ & \centering $32\pi^3/2187$ \tabularnewline
\hline
\centering $6$ &$(5,5,4,3,-5)$ & \centering $-5/36$ & \centering $-32/9$ & \centering $10\pi^3/729$ \tabularnewline
\hline
\centering $6$ &$(5,5,5,-1,-2)$ & \centering $1/9$ & \centering $64/9$ & \centering $16\pi^3/729$ \tabularnewline
\hline
\centering $6$ &$(5,5,5,1,-4)$ & \centering $-2/9$ & \centering $-64/9$ & \centering $32\pi^3/729$ \tabularnewline
\hline
\centering $6$ &$(5,5,5,2,-5)$ & \centering $-5/36$ & \centering $-64/9$ & \centering $20\pi^3/729$ \tabularnewline
\hline
\centering $6$ &$(5,5,5,4,-7)$ & \centering $7/36$ & \centering $64/9$& \centering $28\pi^3/729$ \tabularnewline
\hline
\centering $6$ &$(5,5,5,5,-8)$ & \centering $4/9$ & \centering $64/9$ & \centering $64\pi^3/729$ \tabularnewline
\hline
\end{tabular}
\caption{Volume of $\Pb\stratesp$ when $n=5$.}
\label{table:vols:n5}
\end{table}

\end{document}